\documentclass[12pt, a4paper,reqno]{amsart}

\usepackage{color} \definecolor{bleu_sombre}{rgb}{0,0,0.6}  \definecolor{rouge_sombre}{rgb}{0.8,0,0}\definecolor{vert_sombre}{rgb}{0,0.6,0}
\usepackage[plainpages=false,colorlinks,linkcolor=bleu_sombre,citecolor=rouge_sombre,urlcolor=vert_sombre,breaklinks]{hyperref}

\usepackage[english]{babel}

\usepackage{amsmath,amssymb,amsthm,graphicx,amsfonts,url,color,enumerate,dsfont,stmaryrd,mathabx}

\theoremstyle{plain}
\newtheorem{theorem}{{Theorem}}[section] 
\newtheorem*{theorem*}{{Theorem}}
\newtheorem{proposition}[theorem]{Proposition}
\newtheorem*{proposition*}{Proposition}

\newtheorem*{corollary*}{Corollary}
\newtheorem{lemma}[theorem]{Lemma}
\newtheorem*{lemma*}{Lemma}

\theoremstyle{definition}

\newtheorem*{definition*}{Definition}

\theoremstyle{remark}

\makeatletter

\@addtoreset{equation}{section}  
\makeatother

\newcommand {\limt}[2]{\xrightarrow[#1 \to #2]{}}
\newcommand{\fonc}[4] { \left\{ \begin{array}{ccc} #1 & \to & #2 \\ #3 & \mapsto & #4 \end{array} \right. }

\newcommand{\abs}[1]{\left\vert #1\right\vert}        
\newcommand{\nr}[1]{\left\Vert #1\right\Vert}         
 
\newcommand{\innp}[2]{\left< #1 , #2 \right>}         

\newcommand{\set}[1]{\left\{ #1 \right\}}		
\newcommand{\Ii}[2] {\left\{ #1,\dots,#2 \right\}}
\renewcommand{\leq}{\leqslant}	\renewcommand{\geq}{\geqslant}

\newcommand{\inv}{^{-1}}
\newcommand{\bigo}[2]{\mathop{\Oc}\limits_{#1 \to #2}}

\newcommand{\st}{\,:\,}
         
\renewcommand{\Re}{\mathsf{Re}}        
\renewcommand{\Im}{\mathsf{Im}}

\newcommand{\Dom}{\mathsf{Dom}}

\renewcommand{\ker}{\mathsf{ker}} 

\newcommand{\Id}{\mathsf{Id}}

\newcommand{\R}{\mathbb{R}}		\newcommand{\C}{\mathbb{C}}
\newcommand{\N}{\mathbb{N}}	\newcommand{\Z}{\mathbb{Z}}

\renewcommand{\a}{\alpha}\renewcommand{\b}{\beta}\newcommand{\g}{\gamma}\renewcommand{\d}{\delta}\newcommand{\D}{\Delta}\newcommand{\e}{\varepsilon}\newcommand{\z}{\zeta} \renewcommand{\th}{\theta}\newcommand{\Th}{\Theta}\renewcommand{\k}{\kappa}\renewcommand{\l}{\lambda}\newcommand{\s}{\sigma}\newcommand{\f}{\varphi}\newcommand{\vf}{\phi}\renewcommand{\o}{\omega}\renewcommand{\O}{\Omega}

\newcommand{\Cc}{{\mathcal C}}\newcommand{\Dc}{{\mathcal D}}\newcommand{\Hc}{{\mathcal H}}\newcommand{\Lc}{{\mathcal L}}\newcommand{\Nc}{{\mathcal N}}\newcommand{\Oc}{{\mathcal O}}\newcommand{\Pc}{{\mathcal P}}\newcommand{\Qc}{{\mathcal Q}}\newcommand{\Rc}{{\mathcal R}}\newcommand{\Sc}{{\mathcal S}}\newcommand{\Tc}{{\mathcal T}}\newcommand{\Uc}{{\mathcal U}}\newcommand{\Vc}{{\mathcal V}}\newcommand{\Zc}{{\mathcal Z}}

\newcommand{\stepp}{\noindent {\bf $\bullet$}\quad }

\newcommand{\detail}[1]
{
}

\usepackage{mathrsfs}

\begin{document}

\newcommand{\HH}{\mathscr H}\newcommand{\LL}{\mathscr L}\newcommand{\EE}{\mathscr E} \newcommand{\SSS}{\mathscr S}

\author{Radhia Ayechi, Ilhem Boukhris and Julien Royer}
\address[R. Ayechi]{LAMMDA -- Universit\'e de Sousse -- ESST Hammam Sousse -- Rue Lamine El Abbessi, Hammam Sousse, 4011, Tunisia.}
\email{radhiaayechi@essths.u-sousse.tn}

\address[I. Boukhris]{LAMMDA -- Universit\'e de Sousse -- ESST Hammam Sousse -- Rue Lamine El Abbessi, Hammam Sousse, 4011, Tunisia.}
\email{ilhemboukhris@essths.u-sousse.tn}

\address[J. Royer]{Institut de Math\'ematiques de Toulouse -- UMR 5219 -- Universit\'e Toulouse 3, CNRS -- UPS, F-31062 Toulouse Cedex 9, France.}
\email{julien.royer@math.univ-toulouse.fr}

\title[A system of Schr\"odinger equations in a wave guide]{Energy decay for a system of Schr\"odinger equations in a wave guide}




\begin{abstract}
We prove exponential decay for a system of two Schr\"odinger equations in a wave guide, with coupling and damping at the boundary. This relies on the spectral analysis of the corresponding coupled Schr\"odinger operator on the one-dimensional cross section. We show in particular that we have a spectral gap and that the corresponding generalized eigenfunctions form a Riesz basis.
\end{abstract}

\maketitle

\section{Introduction}

Let $d \geq 2$, $\ell > 0$ and $\O = \R^{d-1} \times ]0,\ell[$. All along the paper, a generic point in $\O$ is denoted by $(x,y)$ with $x \in \R^{d-1}$ and $y \in ]0,\ell[$. We consider on $\O$ a system of Schr\"odinger equations 
\begin{equation} \label{system}
\begin{cases}
i \partial_t u + \D u = 0,\\
i \partial_t v + \D v = 0,
\end{cases}
\quad \text{on } \R_+ \times \O.
\end{equation}
Given $a > 0$ and $b \in \R^*$, damping and coupling are given by the boundary conditions
\begin{equation} \label{system-boundary-conditions}
\begin{cases}
\partial_\nu u(t;x,0) = ia u(t;x,0) + ibv(t;x,0), \\
\partial_\nu v(t;x,0) = - ib u(t;x,0),\\
\partial_\nu u(t;x,\ell) = \partial_\nu v (t;x,\ell) = 0,
\end{cases}
\quad \forall t > 0, \forall x \in \R^{d-1}.
\end{equation}
We could similarly consider the problem with damping and/or coupling on both sides of the boundary. This problem is completed with the initial conditions
\begin{equation} \label{system-CI}
u|_{t = 0} = u_0, \quad v|_{t = 0} = v_0,
\end{equation}
where $u_0,v_0 \in L^2(\O)$.\\

We will check that the problem \eqref{system}-\eqref{system-CI} is well posed. If $U = (u,v)$ is a solution, then for $t \geq 0$ we consider the energy
\[
E(t;U) = \nr{u(t)}_{L^2(\O)}^2 + \nr{v(t)}_{L^2(\O)}^2.
\]
A straightforward computation shows that $E$ is a non-increasing function of time:
\[
\frac d {dt} E(t;U) = -2a \int_{\R^{d-1}} \abs{u(t;x,0)}^2 \, dx \leq 0.
\]
In this paper we are interested in the decay of this energy when the time $t$ goes to $+\infty$.

There is already a large litterature about the energy decay for the Schr\"odinger equation or for the closely related damped wave equation, on compact or non-compact domains, with damping in the domain or at the boundary.

For the damped wave equation on a compact domain, it is known that as soon as we have effective damping in an open subset of the domain or of its boundary, then the energy goes to 0 \cite{Haraux85,Lebeau96}. The decay is uniform with respect to the initial condition (and hence exponential) if and only if all the trajectories of the corresponding classical problem (the rays of light) go through the damping region (see \cite{RauchTay74} for the damping in the domain and \cite{BardosLebRau89} for the damping at the boundary). Otherwise we have at least logarithmic decay with a loss of regularity \cite{Lebeau96,LebeauRob97}. There are intermediate rates of decay when the set of undamped classical trajectories is small and the classical flow is unstable near these trajectories (see for instance \cite{BurqHit07,AnantharamanLea14}).

Similar results have been proved for the (undamped) Schr\"odinger and wave equations in unbounded domains. In this case we look at the energy on a compact subset. It goes to zero if the energy escapes to infinity. The contribution of high frequencies behaves as above. The local energy always goes to 0, at least with logarithmic decay and loss of regularity, and uniformly with decay faster than any negative power of $t$ if and only if all the classical trajectories go to infinity. On the other hand the local energy of the contribution of low frequencies always decays uniformly, with a polynomial rate of decay. We refer for instance to \cite{LaxPhillips67,Ralston69,Burq98,Bouclet11,BonyHaf12}.

We can also consider the damped Schr\"odinger or wave equation in an unbounded domain. If we are interested in the local energy decay, then the geometric condition for high frequencies is that all classical trajectories should go either to infinity or through the damping region. The contribution of low frequencies behaves as in the undamped case if the damping is localized (see for instance \cite{AlouiKhe02,AlouiKhe07,AlouiKhe10,BoucletRoy14,royer-dld-energy-space}), while it behaves like the solution of some heat equation if the damping is effective at infinity (see \cite{Matsumura76,MallougRoy18,JolyRoy18} and references therein).

Here we are interested in the global energy for damped Schr\"odinger equations on a wave guide with damping at infinity. The case of a single Schr\"odinger equation was discussed in \cite{royer-diss-schrodinger-guide}. In that case it was already remarkable that the energy decays exponentially without the geometric control condition for the contribution of high frequencies. Here we have two equations, and only the first is damped (and again, the classical trajectories parallel to the boundary never see the damping). Moreover, the coupling itself is supported by the boundary and does not satisfy the geometric control condition. However, we observe that the energy of both components $u$ and $v$ goes to 0, and furthermore the decay is uniform and hence exponential. Our main result in this paper is the following.

\begin{theorem} \label{th-exp-decay}
Let $a > 0$ and $b \in \R^*$. There exist $\gamma > 0$ and $C > 0$ such that for $(u_0,v_0) \in L^2(\O) \times L^2(\O)$ and $t > 0$ we have 
\[
\nr{u(t)}_{L^2(\O)} + \nr{v(t)}_{L^2(\O)} \leq C e^{-\g t} \big( \nr{u_0}_{L^2(\O)} + \nr{v_0}_{L^2(\O)}\big), 
\]
where $(u,v)$ is the solution of \eqref{system}-\eqref{system-CI}.
\end{theorem}

The proof of Theorem \ref{th-exp-decay} is based on the spectral properties of the corresponding coupled Schr\"odinger operator. For $a,b \in \R$ we consider on $\HH = L^2(\O) \times L^2(\O) \simeq L^2(\O,\C^2)$ the operator 
\begin{equation} \label{def:Pc}
\Pc_{a,b} = \begin{pmatrix} -\D & 0 \\ 0 & -\D \end{pmatrix},
\end{equation}
defined on the subspace $\Dom(\Pc_{a,b})$ of functions $U = (u,v)$ in $H^2(\O;\C^2)$ such that, for all $x \in \R^{d-1}$,
\begin{equation} \label{dom-Pc}
\begin{cases}
\partial_\nu u(x,0) = ia u(x,0) + ibv(x,0),\\
\partial_\nu v(x,0) = - ib u(x,0),\\
\partial_\nu u (x,\ell) = \partial_\nu v(x,\ell) = 0.
\end{cases}
\end{equation}
We will check in Section \ref{sec-general-properties} that if $a \geq 0$ then $\Pc_{a,b}$ is a maximal dissipative operator on $\HH$, so by the usual Lummer-Phillips Theorem it generates a contractions semigroup $(e^{-it\Pc_{a,b}})_{t \geq 0}$ on $\HH$. Then, given $U_0 = (u_0,v_0) \in \Dom(\Pc_{a,b})$, we see that $U = (u,v)$ satisfies \eqref{system}-\eqref{system-CI} if and only if 
\begin{equation} \label{eq:Pc}
\begin{cases}
i\partial_t U - \Pc_{a,b} U = 0,\\
U|_{t = 0} = U_0.
\end{cases}
\end{equation}
Moreover, if $\HH$ is endowed with the natural norm, Theorem \ref{th-exp-decay} is equivalent to the uniform exponential decay in $\Lc(\HH)$ of the propagator $e^{-it\Pc_{a,b}}$ when $t \to +\infty$.

\begin{theorem} \label{th-exp-decay-Pc}
Let $a > 0$ and $b \in \R^*$. There exist $\gamma > 0$ and $C > 0$ such that for $t > 0$ we have 
\[
\nr{e^{-it\Pc_{a,b}}}_{\Lc(\HH)} \leq C e^{-\g t}.
\]
\end{theorem}

By the Gearhart-Pr\"uss Theorem (see for instance Theorem V.1.11 in \cite{engel}), we get uniform exponential decay for $e^{-it\Pc_{a,b}}$ if we can prove that the resolvent of $\Pc_{a,b}$ is well defined and uniformly bounded in an open upper half-plane which contains the real axis. Thus, Theorem \ref{th-exp-decay-Pc} can be deduced from the following spectral result.\\

We denote by $\rho(\Pc_{a,b})$ and $\s(\Pc_{a,b})$ the resolvent set and the spectrum of $\Pc_{a,b}$, respectively.

\begin{theorem} \label{th-spectral-gap-Pc}
Let $a > 0$ and $b \in \R^*$.
\begin{enumerate} [\rm(i)]
\item There exists $\gamma_1 > 0$ such that any $z \in \C$ with $\Im(z) > -\g_1$ belongs to $\rho(\Pc_{a,b})$.
\item There exist $m\in \N^*$ and $C_1 > 0$ such that for $z \in \rho(\Pc_{a,b})$ we have 
\[
\nr{(\Pc_{a,b}-z)\inv}_{\Lc(\HH)} \leq \frac {C_1}{\mathsf{dist}(z,\s(\Pc_{a,b}))^{[m]}},
\]
where for $s > 0$ we have set $s^{[m]} = \min(s,s^m)$.
\end{enumerate}
\end{theorem}

As usual in a wave guide, we use separation of variables and deduce the spectral properties of $\Pc_{a,b}$ on $\O$ from the analogous properties for the corresponding operators on $\R^{d-1}$ and on the cross section $]0,\ell[$. We will split the Laplacian $\Pc_{a,b}$ on $\O$ as the sum of the usual Laplace operator on the first $(d-1)$ variables, and a Laplacian with boundary conditions on $]0,\ell[$.\\

The damping and the coupling are encoded in the domain of the transverse operator. For $a,b \in \R$ we consider on $\SSS = L^2(0,\ell;\C^2)$ the operator 
\begin{equation} \label{def:Tc}
\Tc_{a,b} = \begin{pmatrix} -\partial^2 & 0 \\ 0 & -\partial^2 \end{pmatrix},
\end{equation}
defined on the domain 
\begin{equation} \label{dom-Tc}
\Dom(\Tc_{a,b}) = \set{U \in H^2(0,\ell;\C^2) \, : \, U'(0) + i M_{a,b} U(0) = 0, \quad U'(\ell) = 0},
\end{equation}
where we have set
\begin{equation} \label{def:M}
M_{a,b} = \begin{pmatrix} a & b \\ -b & 0 \end{pmatrix}.
\end{equation}

\begin{theorem} \label{th-spectrum-Tc}
Let $a > 0$ and $b \in \R^*$. The spectrum of $\Tc_{a,b}$ consists of a sequence $(\l_k)_{k \in \N^*}$ of eigenvalues. Moreover,
\begin{enumerate}[\rm(i)]
\item there exists $m \in \N^*$ such that all the eigenvalues have algebraic multplicity smaller than or equal to $m$,
\item there exists $\g_1 > 0$ such that $\Im(\l_k) \leq -\g_1$ for all $k \in \N$,
\item there exists a Riesz basis of $\SSS$ which consists of generalized eigenfunctions of $\Tc_{a,b}$.
\end{enumerate}
\end{theorem}
 
Here we are mostly interested in the dissipative case, but we will see in Section \ref{sec-general-properties} that the adjoint of $\Tc_{a,b}$ is $\Tc_{-a,b}$, so $\Tc_{a,b}$ is selfadjoint if $a = 0$ (then the eigenvalues are real and there exists an orthonormal basis of eigenfunctions), and for $a < 0$ the properties are the same as for $a > 0$, except that the eigenvalues have positive imaginary parts. To be complete we will also include in the intermediate results the already understood case $a > 0$, $b = 0$. This corresponds to two independant equations, one with Neumann boundary conditions and one with damping at the boundary (see \cite{royer-diss-schrodinger-guide}).\\

Theorem \ref{th-spectrum-Tc} is what we need on the transverse operator to prove Theorem \ref{th-spectral-gap-Pc}, but we can give more precise spectral properties for $\Tc_{a,b}$. For example, we will see that $\Tc_{a,b}$ satisfies a Weyl Law. If $P_0$ is the usual Dirichlet Laplacian on a bounded domain $\O_0$ of $\R^d$, then the standard Weyl Law says that the number $N_0(r)$ of eigenvalues of $P_0$ (counted with multiplicities) not greater than $r$ satisfies 
\[
N_0(r) \mathop{\sim}\limits_{r \to +\infty} \frac {r^{\frac d 2} \o_d \abs{\O_0}}{(2\pi)^d},
\]
where $\o_d$ is the volume of the unit ball in $\R^d$ and $\abs{\O_0}$ is the volume of the domain $\O_0$. This result has been improved and extended in many directions.\\

When $a=b=0$ we have two decoupled Schr\"odinger equations with Neumann boundary conditions, so the eigenvalues of $\Tc_0 = \Tc_{0,0}$ are the $n^2 \pi^2 / \ell^2$, $n \in \N$, and these eigenvalues have multiplicity 2. It is easy to deduce that the number of eigenvalues of $\Tc_0$ (counted with multiplicities) smaller than $r > 0$ is of the form $2\ell \sqrt r / \pi + \Oc(1)$.\\

For $r > 0$ we denote by $N_{a,b}(r)$ the number of eigenvalues of $\Tc_{a,b}$ (counted with multiplicities) with real part smaller than $r$.

\begin{theorem}[Weyl Law] \label{th:Weyl}
Let $a,b \in \R$. We have 
\[
N_{a,b}(r) = \frac {2\ell \sqrt{r}}{\pi} + \bigo r {+\infty} (1).
\]
\end{theorem}

In Theorem \ref{th-spectrum-Tc} we have said that the sequence of multiplicities of the eigenvalues of $\Tc_{a,b}$ is bounded. In fact, the maximum multiplicity $m$ given there is the parameter $m$ which appears in Theorem \ref{th-spectral-gap-Pc}. We can be more precise about these multiplicities.

\begin{proposition} \label{prop:multiplicities}
Let $a,b \in \R$.
\begin{enumerate}[\rm(i)]
\item All the eigenvalues of $\Tc_{0,0}$ have geometric and algebraic multiplicities 2.
\item If $(a,b) \neq (0,0)$ then all the eigenvalues of $\Tc_{a,b}$ are geometrically simple.
\item If $a^2 > 4b^2$ then all the eigenvalues of $\Tc_{a,b}$ are algebraically simple.
\item If $a^2 = 4b^2 \neq 0$ then all the eigenvalues of $\Tc_{a,b}$ have algebraic multiplicity 2.
\item If $a^2 < 4b^2$ then the eigenvalues of $\Tc_{a,b}$ can have algebraic multiplicity 1 or 2. More precisely, there exists a countably infinite subset $\Theta$ in $\R^2$ such that all the eigenvalues of $\Tc_{a,b}$ are simple if and only if $(a,b) \notin \Theta$. If $(a,b) \in \O \setminus \{(0,0)\}$ then $\Tc_{a,b}$ has one eigenvalue of algebraic multiplicity 2, and all the others are simple.
\end{enumerate}
In particular, in Theorem \ref{th-spectral-gap-Pc} we can choose $m=1$ if $a^2 > 4b^2$ or $a^2 < b^2$ and $(a,b) \notin \Theta$, and $m= 2$ otherwise.
\end{proposition}

Finally, we notice that we can see our system of two equations on a line segment as a problem on a graph with two edges and non-standard non-selfadjoint conditions at the common vertex (if we have coupling at both ends, then this gives a graph with two edges which have the same ends). Little is known for general non-selfadjoint quantum graphs (see \cite{Hussein14,HusseinKreSie15} for general properties of non-selfadjoint quantum graphs). We also refer to \cite{RiviereRoy20} for a non-selfadjoint Robin Laplacian on a star-shapped graph. In terms of non-selfadjoint quantum graphs, our analysis concerns a very particular example but provides much more precise spectral properties.

\subsection*{Organization of the paper} After this introduction, we give in Section \ref{sec-general-properties} the basic properties for the operator $\Pc_{a,b}$ on $\O$ and for the transverse operator $\Tc_{a,b}$ on $]0,\ell[$. We discuss in Section \ref{sec-eigenvalues} the localization of the large eigenvalues of $\Tc_{a,b}$. This will give in particular the first two statements of Theorem \ref{th-spectrum-Tc}, Theorem \ref{th:Weyl} and Proposition \ref{prop:multiplicities} for large eigenvalues. In Section \ref{sec-Riesz-basis} we finish the proof of Theorem \ref{th-spectrum-Tc} by proving the Riesz basis property and, in Section \ref{sec-resolvent}, we prove Theorem \ref{th-spectral-gap-Pc} from which Theorem \ref{th-exp-decay-Pc} and hence Theorem \ref{th-exp-decay} follow. Finally, we give in Section \ref{sec-eigenvalues-2} more results about the eigenvalues of the transverse operator, in particular about low frequencies.

\section{General properties of the coupled Schr\"odinger operators} \label{sec-general-properties}

In this section we give the basic properties of the operators $\Pc_{a,b}$ and $\Tc_{a,b}$ defined by \eqref{def:Pc}-\eqref{dom-Pc} and \eqref{def:Tc}-\eqref{dom-Tc}, respectively. We will use the following version of the Trace Theorem.

\begin{lemma} \label{lem:trace}
Let $\e > 0$. There exists $C > 0$ such that for all $u \in H^1(0,\ell)$ we have 
\[
\nr{u}_{L^\infty(0,\ell)}^2 \leq \e \nr{u'}_{L^2(0,\ell)}^2 + C \nr{u}_{L^2(0,\ell)}^2. 
\]
\end{lemma}

\begin{proof}
Let $u \in H^1(0,\ell)$ and let $x_0 \in [0,\ell]$ be such that $\min \abs u = \abs{u(x_0)}$. For any $x \in [0,\ell]$ we have 
\[
\abs{u(x)}^2 \leq \abs{u(x_0)}^2 + \abs{\int_{x_0}^x 2 u(s) u'(s) \, ds} \leq \abs{u(x_0)}^2 + \e \nr{u'}_{L^2(0,\ell)}^2 + \frac {\nr{u}_{L^2(0,\ell)}^2} \e.
\]
Since $\ell \abs{u(x_0)}^2 \leq \nr{u}_{L^2(0,\ell)}^2$, the conclusion follows.
\end{proof}

We recall that an operator $A$ on a Hilbert space $\Hc$ is said to be sectorial if there exist $\g_0 \in \R$ and $\th \in \big[ 0, \frac \pi 2 \big[$ such that $\innp{A\f}{\f}_\Hc$ belongs to the sector $\Sigma_{\g_0,\th} = \set{\z \in \C, \abs{\arg(\z-\g_0)} \leq \th}$ for all $\f \in \Dom(A)$ with $\nr{\f}_{\Hc} = 1$. Then it is said to be maximal sectorial if some (and hence any) $\z \in \C \setminus \Sigma_{\g_0,\th}$ belongs to the resolvent set of $A$. We similarly define dissipative and maximal dissipative operators by replacing the sector $\Sigma_{\g_0,\th}$ by the half-space $\set{\z \in \C, \Im(\z) \leq 0}$.

\begin{proposition} \label{prop:Pc}
Let $a,b \in \R$.
\begin{enumerate}[(i)]
\item 
The operator $\Pc_{a,b}$ is maximal sectorial.
\item If $a \geq 0$ then $\Pc_{a,b}$ is maximal dissipative.
\item The adjoint of $\Pc_{a,b}$ is $\Pc_{-a,b}$.
\end{enumerate}
\end{proposition}

\begin{proof}
\stepp For $U = (u,v) \in \Dom(\Pc_{a,b})$ we have 
\begin{eqnarray*}
\lefteqn{\innp{\Pc_{a,b} U}{U}_\HH = \innp{-\D u}{u}_{L^2(\O)} + \innp{-\D v}{v}_{L^2(\O)}}\\
&& = \nr{\nabla u}_{L^2(\O)}^2 + \nr{\nabla v}_{L^2(\O)}^2 -ia \nr{u(\cdot,0)}_{L^2(\R^{d-1})}^2 - 2b  \Im \innp{u(\cdot,0)}{v(\cdot,0)}_{L^2(\R^{d-1})}.
\end{eqnarray*}
In particular, if $a \geq 0$ we have 
\begin{equation} \label{eq:Im-PUU}
\Im \innp{\Pc_{a,b} U}{U}_\HH = - a \nr{u(\cdot,0)}_{L^2(\R^{d-1})}^2 \leq 0,
\end{equation}
so $\Pc_{a,b}$ is dissipative. In general, with $C$ given by Lemma \ref{lem:trace} applied with $\e = \frac 1{2(1+\abs a+\abs b)}$ we have 
\begin{eqnarray*}
\lefteqn{\Re \innp{\Pc_{a,b} U}{U}_{\HH}}\\
&& \geq \nr{\nabla u}_{L^2(\O)}^2 + \nr{\nabla v}_{L^2(\O)}^2 - \abs b \big(\nr{u(\cdot,0)}_{L^2(\R^{d-1})}^2 + \nr{v(\cdot,0)}_{L^2(\R^{d-1})}^2 \big) \\
&& \geq \frac 1 2 \nr{\nabla u}_{L^2(\O)}^2 +  \frac 12 \nr{\nabla v}_{L^2(\O)}^2 - C \abs b \big( \nr{u}_{L^2(\O)}^2 + \nr{v}_{L^2(\O)}^2 \big).
\end{eqnarray*}
and 
\[
\abs{\Im \innp{\Pc_{a,b} U}{U}_\HH}\leq \frac 12 \nr{\nabla u}_{L^2(\O)}^2  + C \abs a \nr{u}_{L^2(\O)}^2,
\]
so
\[
\Re \innp{\big(\Pc_{a,b} + C(\abs a+\abs b)\big) U}{U}_{\HH} \geq \abs{\Im \innp{\big(\Pc_{a,b} + C(\abs a+\abs b)\big) U}{U}_{\HH}}.
\]
This proves that $\Pc_{a,b}$ is sectorial with $\g_0 = -C(\abs a+\abs b)$ and $\th = \frac \pi 4$.

\stepp Let $\l < \gamma_0$. For $U \in \Dom(\Pc_{a,b})$ we have 
\begin{eqnarray} 
\nonumber 
\lefteqn{\nr{(\Pc_{a,b}-\l)U}_{\HH}^2}\\
\nonumber
&& = \nr{(\Pc_{a,b}-\gamma_0) U}_\HH^2 + (\gamma_0-\l)^2 \nr{U}_{\HH}^2 + 2 (\gamma_0-\l) \Re \innp{(\Pc_{a,b}-\gamma_0)U}{U}_\HH\\
&& \geq (\gamma_0-\l)^2 \nr{U}_{\HH}^2.\label{eq:Pc-injective}
\end{eqnarray}
In particular, $(\Pc_{a,b}-\l)$ is injective. Now let $F = (f,g) \in \HH$. For $U = (u,v)$ and $\Phi = (\f,\psi)$ in $H^1(\O;\C^2)$ we set
\begin{equation*}
\Qc(U,\Phi) = \innp{\nabla U}{\nabla \Phi}_{\HH} - \l \innp{U}{\Phi}_{\HH} 
-i \innp{M_{a,b} U(\cdot,0)}{\Phi(\cdot,0)}_{L^2(\R^{d-1};\C^2)}.
\end{equation*}
This defines a sesquilinear form on $H^1(\O;\C^2)$. The computation above ensures that it is coercive. By Lemma \ref{lem:trace}, it is also continuous. Then, by the Lax-Milgram Theorem, there exists $U \in H^1(\O;\C^2)$ such that 
\begin{equation} \label{eq:Qc-Phi}
\forall \Phi \in H^1(\O;\C^2), \quad \Qc(U,\Phi) = \innp{F}{\Phi}_{\HH}.
\end{equation}
Applied with $\Phi$ in $C_0^\infty(\O;\C^2)$, this shows that $U$ belongs to $H^2(\O;\C^2)$ and ${(-\D -\l) U} = F$ in the sense of distributions. Then, after an integration by parts,
\[
\forall \Phi \in H^1(\O;\C^2), \quad \innp{\partial_\nu U(\cdot,0) - i M_{a,b} U(\cdot,0)}{\Phi(\cdot,0)}_{L^2(\R^{d-1},\C)} = 0.
\]
This implies that $U$ belongs to $\Dom(\Pc_{a,b})$, so $(\Pc_{a,b}-\l)U = F$ and $(\Pc_{a,b}-\l)$ is surjective. With \eqref{eq:Pc-injective} we see that $(\Pc_{a,b}-\l)\inv$ is bounded on $\HH$, so $\l$ belongs to the resolvent set of $\Pc_{a,b}$. This proves that $\Pc_{a,b}$ is maximal sectorial, and maximal dissipative if $a\geq 0$.

\stepp By direct computation we see that for $U \in \Dom(\Pc_{a,b})$ and $U^* \in \Dom(\Pc_{-a,b})$ we have
\[
\innp{\Pc_{a,b} U}{U^*}_\HH = \innp{U}{\Pc_{-a,b} U^*}_\HH,
\]
so $\Dom(\Pc_{-a,b}) \subset \Dom(\Pc_{a,b}^*)$ and  $\Pc_{a,b}^*$ coincides with $\Pc_{-a,b}$ on $\Dom(\Pc_{-a,b})$. On the other hand, with the same kind of argument as above, we check that if for some $U^* \in \HH$ there exists $F \in \HH$ such that 
\[
\forall U \in \Dom(\Pc_{a,b}), \quad \innp{\Pc_{a,b} U}{U^*}_\HH = \innp{U}{F}_\HH,
\]
then $U^*$ belongs to $\Dom(\Pc_{-a,b})$. This proves that $\Dom(\Pc_{a,b}^*) \subset \Dom(\Pc_{-a,b})$.   Finally we have proved that $\Pc_{a,b}^* = \Pc_{-a,b}$ and the proof of the proposition is complete.
\end{proof}

We also need similar properties for the transverse operator $\Tc_{a,b}$.

\begin{proposition} \label{prop:Tc}
For $a,b \in \R$ the operator $\Tc_{a,b}$ is maximal sectorial on $\Sc$, its spectrum consists of a sequence of isolated eigenvalues with finite multiplicities and $\Tc_{a,b}^* = \Tc_{-a,b}$. If $a \geq 0$ then $\Tc_{a,b}$ is also maximal dissipative. If $a > 0$ and $b \neq 0$, then all the eigenvalues have negative imaginary parts.
\end{proposition}

\begin{proof}
The facts that $\Tc_{a,b}$ is maximal sectorial, maximal dissipative if $a\geq0$, and that $\Tc_{a,b}^* = \Tc_{-a,b}$ are proved as for $\Pc_{a,b}$ in Proposition \ref{prop:Pc}. Since $\Tc_{a,b}$ is maximal sectorial, its resolvent set is not empty. And since $\Dom(\Tc_{a,b})$ is compactly embedded in $\SSS$, its spectrum consists of isolated eigenvalues with finite multiplicities. 

Now assume that $a > 0$ and $b\neq 0$. By dissipativeness, the eigenvalues of $\Tc_{a,b}$ have non-positive imaginary parts. Now assume that $\l \in \R$ and $U = (u,v) \in \Dom(\Tc_{a,b})$ are such that $\Tc_{a,b} U = \l U$. In particular, as in \eqref{eq:Im-PUU} we have 
\[
0 = \Im \innp{\Tc_{a,b} U}{U}_\SSS = -a\abs{u(0)}^2.
\]
This gives $u(0) = 0$ and hence $v'(0) = 0$. We have $-v'' = \l v$ on $[0,\ell]$ and $v'(0) = v'(\ell) = 0$, so there exists $n \in \N$ such that $\l = n^2 \pi^2 / \ell^2$. Then we have $-u'' = (n^2 \pi^2/\ell^2) u$, $u'(\ell) = 0$ and $u(0) = 0$. This implies that $u = 0$. Then $u'(0) = 0$, so $v(0) = 0$, which implies that $v = 0$. Thus, $\l$ is not an eigenvalue of $\Tc_{a,b}$.
\end{proof}

\section{Transverse eigenvalues} \label{sec-eigenvalues}
 
Let $a,b \in \R$. In this section, we give more precise properties about the localization and the multiplicities of the eigenvalues of $\Tc_{a,b}$. When there is no ambiguity we omit the subscripts $a,b$ of all the involved quantities.\\

As usual in this kind of context, it is easier to discuss the square roots of these eigenvalues. We set 
\begin{equation} \label{def:Sigma}
\Zc = \Zc_{a,b} = \set{z \in \C \st z^2 \text{ is an eigenvalue of } \Tc_{a,b}}.
\end{equation}

Let 
\[
\nu = \frac \pi \ell.
\]
As said in introduction, when $a=b=0$ we have two decoupled Schr\"odinger equations with Neumann boundary conditions, so $\Zc_{0,0} = \nu \Z$, and for all $n \in \N$ the eigenvalue $\nu^2 n^2$ has multiplicity 2. An orthonormal basis of eigenfunctions is given by 
\[
\left( \frac 1 {\sqrt{2\ell}} A_j e_{n\nu} \right) _{n \in \N, j \in \{1,2\}},
\]
where $(A_1,A_2)$ is any orthonormal basis of $\C^2$ and $e_{n\nu}(x) = 2 \cos(n\nu x)$ for all $n \in \N$ and $x \in [0,\ell]$.\\

Let
\[
\mu_\pm = \mu_{\pm,a,b} = \frac {a \pm \d} 2,
\quad \text{where} \quad 
\d = 
\begin{cases}
\sqrt{a^2 - 4b^2} & \text{if } a^2 \geq 4b^2,\\
i \sqrt{4b^2 - a^2} & \text{if } a^2 \leq 4b^2.
\end{cases}
\]
For $z \in \C$ and $x \in [0,\ell]$ we set 
\begin{equation} \label{def:ez}
e_z(x) = e^{izx} + e^{2iz\ell} e^{-izx}
\end{equation}
and 
\begin{equation} \label{def:tilde-ez}
\tilde e_z(x) = \frac 1 {2iz} (\ell-x) \big( e^{izx} - e^{2iz\ell} e^{-izx} \big).
\end{equation}
We have 
\begin{equation} \label{eq:ez-tilde-ez}
-e_z'' - z^2 e_z = 0 \quad \text{and} \quad -\tilde e_z'' - z^2 \tilde e_z = e_z. 
\end{equation}

        \detail 
        {
        \[
        (\partial + iM) (A_1 e_z^1)(0) = \frac 2{z-\mu} \frac {\mu + i\ell(z^2 - \mu M)}{2iz} A_1.
        \]
        \[
        e_z^2(x) = - \frac {(\ell-x)^2}{8z^2} e_z(x) + \frac {i(\ell-x)}{8z^3} e_z^*(x)
        \]
        \[
        (-\partial^2 - z^2) e_z^2 = e_z^1, \quad (e_z^2)'(\ell) = 0.
        \]
        }

Notice that 0 is an eigenvalue of $\Tc$ if and only if $b = 0$. In the following proposition we give a characterization of non-zero eigenvalues.

\begin{proposition} \label{prop:eq-z} 
Let $a,b \in \R$ and $z \in \C^*$. 
\begin{enumerate} [\rm(i)]
\item  
$z^2$ is an eigenvalue of $\Tc$ if and only if 
\begin{equation} \label{eq:exp-2izl}
\phi_{-}(z) \phi_{+}(z) = 0,
\end{equation}
where we have set
\begin{equation} \label{def:phi}
\begin{aligned}
\phi_\pm (z) = \phi_{\pm,a,b}(z) 
& = (z-\mu_{\pm})e^{2iz\ell} - (z+\mu_{\pm})\\
& = z(e^{2iz\ell}-1) - \mu_\pm (e^{2iz\ell} + 1).
\end{aligned}
\end{equation}
Moreover, if $a^2 \neq 4b^2$ the functions $\phi_{-}$ and $\phi_{+}$ have no common zero.
\item 
If $\phi_\pm(z) = 0$ we have  
\begin{equation} \label{expr-ker-T}
\ker(\Tc-z^2) = \set{A_1 e_z, A_1 \in \ker(M-\mu_{\pm})}
\end{equation}
and 
\begin{multline} \label{expr-ker-T-2}
\ker\big( (\Tc-z^2)^2 \big)\\
 = \set{A_2 e_z + A_1 \tilde e_z, A_1 \in \ker(M-\mu_{\pm}) \text{ and } (M-\mu_{\pm}) A_2 = \eta_{\pm}(z) A_1},
\end{multline}
where 
\begin{equation} \label{def:eta}
\eta_\pm(z) = \eta_{\pm,a,b}(z) = \frac {\mu_{\pm} + i\ell(z^2-\mu_{\pm}^2)}{2z^2}.
\end{equation}
\end{enumerate}
\end{proposition}

\begin{proof}
\stepp Let $z \in \C^*$ and $U \in H^2(]0,\ell[,\C^2)$. Then we have $-U'' = z^2 U$ if and only if there exist $A,\tilde A \in \C^2$ such that
\begin{equation} \label{expr-u-v}
\forall x \in ]0,\ell[, \quad U(x) = A e^{izx} + \tilde A e^{-izx}.
\end{equation}
Then $U \in \Dom(\Tc)$ if and only if $\tilde A = e^{2iz\ell} A$ and 
\begin{equation} \label{eq:A}
\big( (1+e^{2iz\ell}) M + (1-e^{2iz\ell})z \big) A = 0.
\end{equation}
There exists a non trivial solution $A$ of \eqref{eq:A} if and only if 
\begin{multline} \label{eq:z-1}
z^2 (1-e^{2iz\ell})^2 + az (1-e^{2iz\ell})(1+e^{2iz\ell}) + b^2 (1+e^{2iz\ell})^2\\
= \det \big( (1+e^{2iz\ell}) M + (1-e^{2iz\ell})z  \big) = 0.
\end{multline}
Assume that $b\neq 0$. We have $e^{2iz\ell} \neq 1$ if $z$ is a solution of \eqref{eq:z-1}. Then \eqref{eq:z-1} can be seen as a second order equation in $z$ with $z$-dependant coefficients. The solutions are given by 
\begin{equation} \label{eq:exp-2izl-bis}
z_\pm  = \mu_\pm  \frac {e^{2iz\ell}+1}{e^{2iz\ell}-1},
\end{equation}
and \eqref{eq:exp-2izl} follows. We conclude similarly when $b = 0$. Moreover, since $e^{2iz\ell} + 1 \neq 0$, we see that if $\phi_-(z) = 0 = \phi_+(z)$ we have $\mu_+ = \mu_-$ and hence $a^2 = 4b^2$.

\stepp If $b\neq 0$ we have $z_\pm \neq \mu_\pm$ by \eqref{eq:exp-2izl-bis}. Then we can write 
\[
(1-e^{2iz_\pm \ell}) = -\frac {2\mu_\pm}{z_\pm-\mu_\pm} \quad \text{and} \quad (1+e^{2iz_\pm\ell}) = \frac {2z_\pm}{z_\pm-\mu_\pm},
\]
so \eqref{eq:A} with $z_\pm$ holds if and only if $A \in \ker(M-\mu_\pm)$. This gives \eqref{expr-ker-T}.

Let $\s \in \{-,+\}$ and $U_2 \in \ker((\Tc-z_\s^2)^2)$. Let $U_1 = (\Tc-z_\s^2) U_2 \in \ker(\Tc-z_\s^2)$. By \eqref{expr-ker-T} there exists $A_1 \in \ker(M-\mu_\s)$ such that $U_1 = A_1 e_{z_\s}$. Then, by \eqref{eq:ez-tilde-ez} there exist $A_2,\tilde A_2 \in \C^2$ such that, for all $x \in [0,\ell]$,
\[
U_2(x) = A_1 \tilde e_{z_\s}(x) + A_2 e^{iz_\s x} + \tilde A_2 e^{-iz_\s x}.
\]
Since $U_2'(\ell) = 0$ we necessarily have $\tilde A_2 = e^{2i z_\s \ell} A_2$, so that 
\begin{equation} \label{eq:expr-U2}
U_2 = A_1 \tilde e_{z_\s} + A_2 e_{z_\s}.
\end{equation}
Finally the condition $U_2'(0) + iMU_2(0) = 0$ gives 
\begin{equation} \label{eq:cond-U2}
(M-\mu_\sigma) A_2 = \eta_\s(z_\s) A_1.
\end{equation}
Conversely, if $U_2$ satisfies \eqref{eq:expr-U2} and \eqref{eq:cond-U2} then it belongs to $\ker((\Tc-z_\s^2)^2)$. The proof is similar when $b = 0$.
\end{proof}

\detail{
\[
f'(z) = 2i\ell e^{2iz\ell} + \frac {2\mu}{(z-\mu)^2}
\]

\[
f'(z) = \frac 2 {(z-\mu)^2} \big( \mu + i\ell (z^2-\mu^2) \big)
\]

\[
f''(z) = - 4\ell^2 e^{2iz\ell} - \frac {4\mu}{(z-\mu)^2}
\]

\[
f''(z) = - \frac {4\ell (z+\mu)}{(z-\mu)^2} \big( \ell^2(z-\mu) - i\big)
\]

\[
U_3(x) = - \frac {A_1}{8z^2} (x-\ell)^2 e_z - \frac {iA_1}{8z^3} (x-\ell) e_z^* + \frac {iA_2}{2z} (x-\ell) e_z^* + A_3e_z
\]
}

Now we apply \eqref{eq:exp-2izl} to prove that the large eigenvalues of $\Tc = \Tc_{a,b}$ are close to the eigenvalues of the Neumann decoupled operator $\Tc_0 = \Tc_{0,0}$.

\begin{proposition} \label{prop:Sigma-disks}
Let $a,b \in \R$.
\begin{enumerate}[\rm(i)]
\item Let $\e > 0$. There exists $R_\e > 0$ such that for $s \in [0,1]$ and $z \in \Zc_{sa,sb}$ with $\abs z \geq R_\e$ we have $\mathsf{dist}(z, \nu \Z) < \e$. In particular there exists $n_0 \in \N^*$ such that for all $s \in [0,1]$ we have 
\[
\Zc_{sa,sb} \subset  D(0, R) \cup \bigcup_{\abs n \geq n_0} D \left(n\nu , \frac \nu 6 \right),  \quad \text{where }  R = \left(n_0 - \frac 12 \right)\nu.
\]
\item For $n \geq n_0$ the functions $\phi_-$ and $\phi_+$ have exactly one zero in the disk $D\big( n\nu , \frac \nu 6\big)$.
\end{enumerate}
\end{proposition}

\begin{proof}
\stepp Assume by contradiction that the first statement does not hold. Then there exist sequences $(s_m)_{m \in \N}$ in $[0,1]$, $(z_m)_{m \in \N}$ in $\C$ and $(\s_m)_{m \in \N}$ in $\set\pm$ such that $\phi_{\s_m,s_ma,s_mb}(z_m) = 0$, $\mathsf{dist}(z_m,\nu\Z) \geq \e$ for all $m \in \N$ and $\abs {z_m} \to +\infty$. In particular $z_m \notin \nu\Z$ so $s_m b \neq 0$ and $z_m \neq \mu_m$, where we have set $\mu_m = \mu_{\s_m,s_m a,s_m b}$. Since $\mu_m = s_m \mu_{\s_m,a,b}$ is bounded we have 
\[
e^{2iz_m \ell}  = \frac {z_m +  \mu_m}{z_m- \mu_m}  \limt m {+\infty} 1,
\]
so $\mathsf{dist}(z_m,\nu\Z) \to 0$ and we get a contradiction. The second statement follows by choosing $\e = \frac \nu 6$ (we can take $R_{\nu/6}$ larger to ensure that it is of the required form).

\stepp For $z \in \C$ we set 
\begin{equation} \label{def:phi0}
\phi_0(z) = \phi_{\pm,0,0}(z) = (e^{2iz\ell} - 1)z.
\end{equation}
Then, for all $n \in \N^*$, $\phi_0$ has one simple zero in the disk $D\big(n\nu, \frac \nu 6 \big)$ and does not vanish on its boundary. On the other hand $\phi_\pm$ does not vanish on the circle $C\big(n\nu, \frac \nu 6 \big)$ for $n$ large enough. Choosing $n$ larger if necessary we also have on this circle
\[
\abs{\phi_\pm(z) - \phi_0(z)} = \abs{\mu_\pm} \abs{e^{2iz\ell}+1} < \abs{\phi_0(z)}.
\]
By the Rouch\'e Theorem, $\phi_\pm$ has exactly one simple zero in $D\big(n\nu, \frac \nu 6 \big)$.
\end{proof}

For $n \geq n_0$ we denote by $z_{n,\pm} = z_{n,\pm,a,b}$ the unique solution of $\phi_\pm(z) = 0$ in $D\big(n\nu,\frac \nu 6\big)$. If $a^2= 4b^2$ we can simply write $z_n$ instead of $z_{n,+}$ or $z_{n,-}$.\\

The following proposition gives a rough localization of the high frequency eigenvalues and in particular the Weyl Law for $\Tc$.

\begin{proposition} \label{prop:multiplicities-high-freq}
Let $a,b \in \R$. Let $n_0\in\N^*$ be given by Proposition \ref{prop:Sigma-disks}. The spectrum of $\Tc$ is contained in 
\[
\Dc = D(0,R^2) \cup \bigcup_{n \geq n_0} D\left( n^2 \nu^2 , \frac {n \nu^2} 2 \right).
\]
The sum of the algebraic multiplicities of the eigenvalues of $\Tc$ in $D(0,R^2)$ is $2n_0$ and, for $n \geq n_0$,
\begin{enumerate}[\rm(i)]
\item if $a^2 \neq 4b^2$, then $\Tc$ has exactly two distinct simple eigenvalues in $D\left( n^2 \nu^2 , \frac {n \nu^2} 2 \right)$, given by $z_{n,-}^2$ and $z_{n,+}^2$,
\item if $a^2 = 4b^2$, then $\Tc$ has a unique double eigenvalue in $D\left( n^2 \nu^2 , \frac {n \nu^2} 2 \right)$, given by $z_n^2$.
\end{enumerate}
In particular, for $r > 0$ we have 
\[
\frac {2\sqrt r}{\nu} - 1 \leq N(r) \leq \max \left( 2n_0, \frac {2\sqrt r}\nu + 3 \right).
\]
\end{proposition}

\begin{proof}
Let $n \geq n_0$. The operator $\Tc_s = \Tc_{sa,sb}$ is analytic with respect to the parameter $s$ in the sense of Kato (family of type B, see \cite{kato}). By Proposition \ref{prop:Sigma-disks}, the circle $C\big( n^2 \nu^2 , \frac {n \nu^2} 2 \big)$ is included in the resolvent set of $\Tc_s$ for all $s \in [0,1]$, so the number of eigenvalues (counted with multiplicities) of $\Tc_s$ in the disk $D_n = D\big( n^2 \nu^2 , \frac {n \nu^2} 2 \big)$ is independant of $s \in [0,1]$. Since $\Tc_0$ has exactly one double eigenvalue in $D_n$, the number of eigenvalues (counted with multiplicities) of $\Tc = \Tc_1$ in $D_n$ is also equal to 2.

If $a^2 \neq 4b^2$, we already know that $\Tc$ has two distinct eigenvalues $z_{n,-}^2$ and $z_{n,+}^2$ in $D_n$. They are necessarily simple. 

If $a^2 = 4b^2$, then $z_n^2$ is the unique eigenvalue of $\Tc$ in $D_n$. It has algebraic multiplicity 2 (we know from Proposition \ref{prop:eq-z} that it is geometrically simple if $(a,b) \neq (0,0)$).

Similarly, $\Tc$ has the same number of eigenvalues as $\Tc_0$ in $D(0,R^2)$, that is $2n_0$ (counted with multiplicities).

Now let $r > 0$. If $r \leq R^2 =  \big(n_0-\frac 12\big)^2\nu^2$, then $N(R) \leq 2n_0$. Otherwise, there exists $n \geq n_0$ such that $r \in \big[\big(n-\frac 12\big)^2 \nu^2, \big(n+\frac 12\big)^2 \nu^2 \big]$. Then we have 
\[
\frac {2\sqrt r}\nu - 1 \leq 2n \leq N(r) \leq 2n+2 \leq \frac {2\sqrt r} \nu + 3,
\]
which concludes the proof.
\end{proof}

We finish this section by an asymptotic expension for the large eigenvalues of $\Tc$. We already know that all the eigenvalues have negative imaginary parts when $a>0$ and $b\neq 0$, so this asymptotics gives in particular a spectral gap for $\Tc$ in this case.

\begin{proposition} \label{prop:DL}
Let $a,b \in \R$. We have 
\[
z_{n,\pm} = n\nu - \frac {i\mu_\pm}{n\pi} + \bigo n {+\infty} (n^{-2}),
\]
and hence 
\[
z_{n,\pm}^2 = n^2 \nu^2 - \frac {2i\mu_\pm}{\ell} + \bigo n {+\infty} (n^{-1}).
\]
\end{proposition}

\begin{proof} 
Let $r_{n,\pm}  := z_{n,\pm} - n\nu$. By Proposition \ref{prop:Sigma-disks}, $r_{n,\pm}$ goes to 0 as $n \to +\infty$. Then we compute 
\begin{align*}
0  = \phi_\pm(z_{n,\pm})
& = (n\nu + r_{n,\pm} - \mu_\pm) e^{2ir_{n,\pm} \ell} - (n\nu + r_{n,\pm} + \mu_\pm)\\
& = 2ir_{n,\pm} n\pi - 2\mu_\pm + \Oc(r_{n,\pm}) + \Oc(n r_{n,\pm}^2).
\end{align*}
This proves that 
\[
r_{n,\pm} = - \frac {i\mu_\pm}{n\pi} + \Oc(n^{-2}),
\]
and the conclusions follow.
\end{proof}

\section{Riesz basis property} \label{sec-Riesz-basis}

For $b \in \R$ the operator $\Tc_{0,b}$ is selfadjoint with compact resolvent, so there exists an orthonormal basis of eigenfunctions for $\Tc_{0,b}$.\\

In this section we consider $(a,b) \in \R^* \times \R$ and we prove that there exists a Riesz basis of $\SSS$ made with generalized eigenfunctions of $\Tc = \Tc_{a,b}$. Since $\Tc$ is not selfadjoint, we do not necessarily have a Hilbert basis of eigenvectors. However, we will prove that the generalized eigenvectors are not too far, in a suitable sense, from being orthogonal. We have no control on these vectors for $n$ small, but when $n$ is large we will see that they are in fact close to a family of eigenvectors of $\Tc_{0,0}$ which forms a Riesz basis.\\ 

We recall that a family $(\Psi_k)_{k \in \N^*}$ in $\SSS$ is a Riesz basis if the operator 
\[
\fonc{\ell^2(\N^*)}{\SSS}{(u_k)_{k \in \N^*}}{\sum_{k=1}^\infty u_k \Psi_k}
\]
is well defined, bounded, bijective and boundedly invertible. Then for any $U \in \SSS$ there is a unique sequence $(u_k)_{k\in\N^*} \in \ell^2(\N^*)$ such that $U = \sum_{k=1}^\infty u_k \Psi_k$ and, for some $C \geq 1$ independant of $U$,
\begin{equation} \label{eq:Riesz-basis}
C\inv \sum_{k=1}^\infty \abs{u_k}^2 \leq \nr{U}_\SSS^2 \leq C \sum_{k=1}^\infty \abs{u_k}^2.
\end{equation}
For general results about Riesz bases we refer for instance to \cite{agranovich94}.\\

If $a^2 \neq 4b^2$ then for all $n \geq n_0$ ($n_0 \in \N^*$ given by Proposition \ref{prop:Sigma-disks}) we set 
\begin{equation} \label{def:Phi-neq}
\Phi_{2n+1} = A_- e_{z_{n,-}} \quad \text{and} \quad \Phi_{2n+2} = A_+ e_{z_{n,+}},
\end{equation}
where 
\begin{equation} \label{def:A}
A_\pm = \begin{pmatrix} \mu_\pm \\ -b \end{pmatrix} \in \ker(\Tc-\mu_\pm)
\end{equation}
(we omit the subscripts $a,b$). Then $\Phi_{2n+1},\Phi_{2n+2} \in \Dom(\Tc)$, $\Tc \Phi_{2n+1} = z_{n,-}^2 \Phi_{2n+1}$ and $\Tc \Phi_{2n+2} = z_{n,+}^2 \Phi_{2n+2}$.

Similarly, if $a^2 = 4b^2$ we set for $n \geq n_0$
\begin{equation} \label{def:Phi-eq}
\Phi_{2n+1} = A_1 e_{z_n}, \quad \Phi_{2n+2} = A_1 \tilde e_{z_n} + A_{2,n} e_{z_n}, 
\end{equation}
where 
\[
A_1 = \begin{pmatrix} \frac a 2 \\ - b \end{pmatrix}  \quad \text{and} \quad A_{2,n} = \begin{pmatrix} \eta(z_n)  \\ 0 \end{pmatrix}.
\]
We recall that the parameter $\eta(z_n)$ is defined in \eqref{def:eta} ($\mu_+ = \mu_-$ in this case). We have 
\begin{equation} \label{eq:lim-A2}
A_{2,n} = A_{2,\infty} + \Oc(n\inv), \quad \text{where} \quad  A_{2,\infty} =  \begin{pmatrix} \frac {i\ell} 2 \\ 0 \end{pmatrix}.
\end{equation}
In particular, choosing $n_0$ larger if necessary, we can assume that $\eta(z_n) \neq 0$ for $n \geq n_0$. Then we have $\Phi_{2n+1},\Phi_{2n+2} \in \Dom(\Tc)$, $\Tc \Phi_{2n+1} = z_n^2 \Phi_{2n+1}$ and $\Tc \Phi_{2n+2} = z_n^2 \Phi_{2n+2} + \Phi_{2n+1}$.

In both cases, we also consider a basis $(\Phi_k)_{1 \leq k \leq 2n_0}$ of generalized eigenfunctions for the subspace of $\SSS$ spaned by the generalized eigenspaces corresponding to the eigenvalues of $\Tc$ in $D(0,R^2)$ (see Proposition \ref{prop:multiplicities-high-freq}).

All this defines a family $(\Phi_k)_{k \in \N^*}$ of generalized eigenfunctions for $\Tc$. Our purpose is to prove that this is a Riesz basis of $\SSS$. Notice that we could normalize these vectors, but this is not necessary (\eqref{eq:norm-Phi} below is enough).\\

We begin with the following computation.

\begin{lemma} 
\begin{enumerate}[\rm(i)]

\item We have 
\begin{equation} \label{eq:norm-ez}
\nr{e_z}_{L^2(0,\ell)}^2 \xrightarrow[\substack{\abs{\Re(z)} \to +\infty \\ \Im(z) \to 0}]{} 2\ell.
\end{equation}

\item There exists $c > 0$ such that for $z \in \C$ with $\abs{\Im(z)} \leq 1$ we have 
\begin{equation} \label{eq:norm-etildez}
\|\tilde e_z\|_{L^2(0,\ell)} \leq  \frac c {\abs z}.
\end{equation}

\item For $z,\z \in \C$ with $z \neq \pm \bar \z$ we have 
\begin{equation} \label{eq:innp-ez-ez}
\innp{e_z}{e_\z}_{L^2(0,\ell)} = \frac {e^{2i(z-\bar \z)\ell}-1}{i(z-\bar \z)} + \frac {e^{2iz\ell} - e^{-2i\bar \z \ell}}{i(z+\bar \z)}.
\end{equation}
and
\begin{multline} \label{eq:innp-ez-tilde-ez}
\innp{e_z}{\tilde e_\z}_{L^2(0,\ell)}  =\\
-\frac {\ell(e^{2i(z-\bar \z)\ell}+1)}{2 \bar \z (z-\bar \z)} 
+ \frac {e^{2i(z-\bar \z)\ell} - 1}{2i\bar \z (z-\bar \z)^2} 
+ \frac {\ell(e^{2iz\ell}+e^{-2i\bar \z\ell})}{2\bar \z (z+\bar \z)} 
- \frac {e^{2iz\ell} - e^{-2i\bar \z \ell}}{2i\bar \z (z+\bar \z)^2}.
\end{multline}
\end{enumerate}\label{lem:computation-ez}
\end{lemma}

\begin{proof}
We check \eqref{eq:innp-ez-ez} and \eqref{eq:innp-ez-tilde-ez} by direct computation. If $z \notin \R$ we can take $\z = z$ in \eqref{eq:innp-ez-ez}. This gives 
\[
\nr{e_z}_{L^2(0,\ell)}^2 =  e^{-2 \Im(z) \ell} \left( \frac {\sinh(2 \Im(z)\ell)}{\Im(z)}+ \frac { \sin\big( 2 \Re(z) \ell \big)}{\Re(z)} \right),
\]
and \eqref{eq:norm-ez} follows. We get the same conclusion if $z \in \R$. Finally, \eqref{eq:norm-etildez} is clear from the definition of $\tilde e_z$.
\end{proof}

We can first consider separately the vectors $\Phi_k$ for $1\leq k \leq 2n_0$, and then the vectors $\Phi_{2n+1}$ and $\Phi_{2n+2}$ for each $n \geq n_0$. It is obvious that a finite family of linearly independant vectors is a Riesz basis of the finite dimensional subspace that it spans. The interest of the following proposition is that we can take the same constant for each of these finite subfamilies of eigenfunctions.

\begin{proposition} \label{prop:intermediate-bases}
There exists $C > 0$ such that for all $(u_k)_{k \in \N^*} \in \ell^2(\N^*)$ we have 
\[
C\inv \sum_{k=1}^{2n_0} \abs{u_k}^2 \leq \nr{U_0}_{\SSS}^2 \leq C \sum_{k=1}^{2n_0} \abs{u_k}^2
\]
and 
\[
C\inv \big( \abs{u_{2n+1}}^2 + \abs{u_{2n+2}}^2 \big) \leq \nr{U_n}_{\SSS}^2 \leq C \big( \abs{u_{2n+1}}^2 + \abs{u_{2n+2}}^2 \big),
\]
where we have set $U_0 = \sum_{k=1}^{2n_0} u_k \Phi_k$ and, for $n \geq n_0$, $U_n = u_{2n+1} \Phi_{2n+1} + u_{2n+2} \Phi_{2n+2}$.
\end{proposition}

\begin{proof}
\stepp The vectors $(\Phi_1,\dots,\Phi_{2n_0})$ are linearly independant by definition, and we see directly that for $n \geq n_0$ the vectors $\Phi_{2n+1}$ and $\Phi_{2n+2}$ are not colinear. It remains to check that we have the estimates with a constant independant of $n$.

\stepp Assume that $a^2 \neq 4b^2$. By Lemma \ref{lem:computation-ez} we have
\[
\nr{\Phi_{2n+1}}_\SSS^2 = \abs{A_-}^2_{\C^2} \nr{e_{z_{n,-}}}_{L^2(0,\ell)}^2 \limt n {+\infty} 2\ell \abs{A_-}_{\C^2}^2,
\]
and similarly $\nr{\Phi_{2n+2}}_\SSS^2$ goes to $2\ell \abs{A_+}_{\C^2}^2$. Thus there exists $C_1 \geq 1$ such that, for all $k \in \N^*$,
\begin{equation} \label{eq:norm-Phi}
C_1\inv \leq \nr{\Phi_k}_\SSS^2 \leq C_1.
\end{equation}
In particular,
\begin{equation}
\begin{aligned}
\nr{U_n}_\SSS^2 
& \leq 2 \abs{u_{2n+1}}^2 \nr{\Phi_{2n+1}}_\SSS^2 + 2 \abs{u_{2n+2}}^2 \nr{\Phi_{2n+2}}_\SSS^2\\
& \leq 2C_1 \big( \abs{u_{2n+1}}^2 + \abs{u_{2n+2}}^2 \big).
\end{aligned}
\end{equation}

\stepp Since $A_-$ and $A_+$ are not colinear in $\C^2$ there exists $\e > 0$ such that 
\[
\abs{\innp{A_-}{A_+}_{\C^2}} \leq (1-\e) \abs{A_-}_{\C^2}\abs{A_+}_{\C^2}.
\]
Then
\begin{align*}
\abs{\innp{\Phi_{2n+1}}{\Phi_{2n+2}}_\SSS}
& = \abs{\innp{A_-}{A_+}_{\C^2}} \abs{\innp{e_{z_{n,-}}}{e_{z_{n,+}}}_{L^2(0,\ell)}}\\
& \leq (1-\e)\abs{A_-}_{\C^2}\abs{A_+}_{\C^2} \nr{e_{z_{n,-}}}_{L^2(0,\ell)} \nr{e_{z_{n,+}}}_{L^2(0,\ell)}\\
& \leq (1-\e) \nr{\Phi_{2n+1}}_\SSS \nr{\Phi_{2n+2}}_\SSS,
\end{align*}
and hence 
\begin{align*}
\nr{U_n}_\SSS^2 
& \geq \abs{u_{2n+1}}^2 \nr{\Phi_{2n+1}}_\SSS^2 + \abs{u_{2n+2}}^2 \nr{\Phi_{2n+2}}_\SSS^2 \\
& \quad - 2 \abs{u_{2n+1}} \abs{u_{2n+2}} \abs{\innp{\Phi_{2n+1}}{\Phi_{2n+2}}_\SSS} \\
& \geq \e \big( \abs{u_{2n+1}}^2 \nr{\Phi_{2n+1}}_\SSS^2 + \abs{u_{2n+2}}^2 \nr{\Phi_{2n+2}}_\SSS^2 \big)\\
& \geq \e C_1\inv \big( \abs{u_{2n+1}}^2 + \abs{u_{2n+2}}^2 \big).
\end{align*}
The proof is complete if $a^2 \neq 4b^2$.

\stepp Now assume that $a^2 = 4b^2$. By Lemma \ref{lem:computation-ez} and \eqref{eq:lim-A2} we have 
\[
\nr{\Phi_{2n+1}}_\SSS^2 \limt n {+\infty} 2\ell \abs{A_1}_{\C^2}^2, \quad \nr{\Phi_{2n+2}}_\SSS^2 \limt n {+\infty} 2\ell \abs{A_{2,\infty}}_{\C^2}^2.
\]
Since $A_1$ and $A_{2,\infty}$ are not colinear, there exists $\e > 0$ such that for $n$ large enough we have $\abs{\innp{A_1}{A_{2,\infty}}} \leq (1-\e) \abs{A_1}\abs{A_{2,\infty}}$. We can proceed as in the previous case.
\end{proof}

Now we consider pairs of generalized eigenfunctions which are not associated to the same subfamilies of eigenfunctions.\\

We set $\SSS_0 = \mathsf{span}(\Phi_1,\dots,\Phi_{2n_0})$ and, for $n \geq n_0$, $\SSS_n = \mathsf{span}(\Phi_{2n+1},\Phi_{2n+2})$.

\begin{proposition} \label{prop:innp-SSSn}
There exists $C > 0$ such that for $n \geq n_0$, $m \in \N^*$, $U_n \in \SSS_n$ and $U_{n+m} \in \SSS_{n+m}$ we have 
\[
\abs{\innp{U_n}{U_{n+m}}_\SSS} \leq \frac {C \nr{U_n}_\SSS \nr{U_{n+m}}_\SSS}{nm}.
\] 
\end{proposition}

\begin{proof}
We begin with the case $a^2 \neq 4b^2$. We apply \eqref{eq:innp-ez-ez} with $z = z_{n,-}$ and $\z = z_{n+m,-}$. With Proposition \ref{prop:DL} we get
\[
\abs{\innp{\Phi_{2n+1}}{\Phi_{2(n+m)+1}}}  = \abs{A_-}^2  \abs{\big<e_{z_{n,-}} , e_{z_{n+m,-}} \big>_{L^2(0,\ell)}} \lesssim \frac {1}{nm}
\]
and then 
\[
\abs{\innp{u_{2n+1} \Phi_{2n+1}}{ u_{2(n+m)+1} \Phi_{2(n+m)+1}}} 
\lesssim  \frac {\nr{U_n}_\SSS \nr{U_{n+m}}_\SSS}{nm}.
\]

We estimate similarly $\abs{\innp{\Phi_{2n+1}}{\Phi_{2m+2}}}$, $\abs{\innp{\Phi_{2n+2}}{\Phi_{2m+1}}}$ and $\abs{\innp{\Phi_{2n+2}}{\Phi_{2m+2}}}$, and the conclusion follows in this case.
For the case $a^2 = 4b^2$ we proceed similarly, now using \eqref{eq:innp-ez-ez}, \eqref{eq:innp-ez-tilde-ez} and the fact that by \eqref{eq:norm-etildez} we have 
\[
\abs {\innp{\tilde e_{z_n}}{\tilde e_{z_{n+m}}}} \lesssim \frac 1 {n(n+m)} \lesssim \frac 1 {nm}. \qedhere
\]
\end{proof}

To prove that the family $(\Phi_k)_{k \in \N}$ is a Riesz basis, we compare it with a family of eigenfunctions for the model operator $\Tc_0$. 

Assume that $a^2 \neq 4b^2$. For $n \in \N$ we set 
\begin{equation} \label{def:Phi-0}
\Phi_{2n+1}^0  = A_- e_{n\nu} \quad \text{and} \quad \Phi_{2n+2}^0  = A_+ e_{n\nu}.
\end{equation}
Since $(A_-,A_+)$ is a basis of $\C^2$, there exists $C_0 \geq 1$ such that for $V \in \C^2$ and the unique $(v_-,v_+) \in \C^2$ such that $V = v_- A_- + v_+ A_+$ we have 
\[
C_0 \inv \big( \abs{v_-}^1 + \abs{v_+}^2 \big) \leq \nr{V}_{\C^2}^2 \leq C_0 \big( \abs{v_-}^1 + \abs{v_+}^2 \big).
\]
Let $U \in \SSS$. There exist $u_- ,u_+ \in L^2(0,\ell)$ unique such that for almost all $y \in ]0,\ell[$
\[
U(y) = u_-(y) A_- + u_+(y) A_+.
\]
Since $(e_{n\nu})_{n \in \N}$ is an orthonormal basis of $L^2(0,\ell)$, there exist unique sequences $(u_n^-)_{n \in \N}$ and $(u_n^+)_{n \in \N}$ in $\ell^2(\N^*)$ such that 
\[
U(y) = \sum_{n \in \N} u_n^- e_{n\nu}(y) A_- + \sum_{n \in \N} u_n^+ e_{n\nu}(y) A_+.
\]
Moreover,
\[
C_0\inv \sum_{n \in \N} \big( \abs{u_n^-}^2 + \abs{u_n^+}^2 \big) \leq \nr{U}_\SSS^2 \leq C_0 \sum_{n \in \N} \big( \abs{u_n^-}^2 + \abs{u_n^+}^2 \big).
\]
This means that the family $(\Phi_k^0)_{k \in \N^*}$ is a Riesz basis of $\SSS$.
Now assume that $a^2 = 4b^2$. Then we get the same conclusion if we set 
\begin{equation} \label{def:Phi-0-eq}
\Phi_{2n+1}^0  = A_1 e_{n\nu} \quad \text{and} \quad \Phi_{2n+2}^0  = A_{2,\infty} e_{n\nu}.
\end{equation}

\begin{proposition} \label{prop:Riesz-basis}
The family $(\Phi_k)_{k \in \N^*}$ is a Riesz basis of $\SSS$.
\end{proposition}

\begin{proof}
\stepp We consider the map 
\[
\Th : \fonc{\ell^2(\N^*)}{\SSS} {(u_k)_{k \in \N}}{\sum_{k=1}^{+\infty} u_k \Phi_k}
\]
and we prove that $\Th$ is well defined, continuous, bijective and has continuous inverse.

\stepp  Let $(U_n)_{n \in \Nc}$ be a sequence in $\SSS$ with $U_n \in \SSS_n$ for all $n \in \Nc =\set 0 \cup \{n_0,n_0+1,\dots\}$ and $\sum_{n \in \Nc} \nr{U_n}_\SSS^2 < +\infty$. For $N \geq n_0$ and $p \in \N$ we have
\[
\nr{\sum_{n=N}^{N+p} U_n}_\SSS^2 - \sum_{n=N}^{N+p} \nr{U_n}_\SSS^2 = \sum_{n=N}^{N+p} \sum_{m=1}^{N+p-n} 2 \Re \innp{U_n}{U_{n+m}}_\SSS,
\]
so by Proposition \ref{prop:innp-SSSn}
\begin{equation} \label{eq:almost-orthogonal}
\begin{aligned}
\abs{\nr{\sum_{n=N}^{N+p} U_n}_\SSS^2 - \sum_{n=N}^{N+p} \nr{U_n}_\SSS^2}
& \lesssim  \sum_{n=N}^{+\infty} \frac {\nr{U_n}_\SSS}{n} \sum_{m=1}^{+\infty} \frac {\nr{U_{n+m}}_\SSS} {m}\\
& \lesssim \sqrt{\sum_{n = N}^{+\infty} \frac 1 {n^2}} \sum_{n=N}^{+\infty} \nr{U_n}_\SSS^2.
\end{aligned}
\end{equation}
This proves in particular that the series $\sum_{n \in \Nc} U_n$ converges in $\SSS$ and that there exists $C > 0$ independant of the sequence $(U_n)$ such that 
\[
\nr{\sum_{n \in \Nc} U_n}_\SSS^2 \leq C \sum_{n \in \Nc} \nr{U_n}_\SSS^2.
\]
With Proposition \ref{prop:intermediate-bases}, this proves that $\Th$ is well defined and bounded.

\stepp 
For $n \geq n_0$ we set
\[
\Pi_n = - \frac 1 {2i\pi} \int_{\Cc \big(n^2\nu^2,\frac {n\nu^2}2\big)} (\Tc-\z)\inv \, d\z \quad \in \Lc(\SSS).
\]
Then $\SSS_n = \ker(\Pi_n - \Id_\SSS)$, and $\SSS_j \subset \ker(\Pi_n)$ for all $j \in \Nc \setminus \set n$.
Now assume that the sequence $u = (u_k)_{k \in \N}$ is such that $\Th(u) = 0$. For $n \in \N$ we have $u_{2n+1} \Phi_{2n+1} + u_{2n+2} \Phi_{2n+2} = \Pi_n(\Th(u)) = 0$, so $u_{2n+1} = u_{2n+2} = 0$. Since this holds for all $n \geq n_0$ we have $\sum_{k=1}^{2n_0} u_k \Phi_k = 0$, and then $u_1 = \dots = u_{2n_0} = 0$. This proves that $\Th$ is injective.

\stepp By \eqref{eq:almost-orthogonal} there exists $N \geq n_0$ independant of $U$ such that 
\begin{equation} \label{eq:almost-orthogonal-2}
\abs{\nr{\sum_{n=N}^{N+p} U_n}_\SSS^2 - \sum_{n=N}^{N+p} \nr{U_n}_\SSS^2}
 \leq  \frac 12 \sum_{n=N}^{+\infty} \nr{U_n}_\SSS^2.
\end{equation}
Assume by contradiction that there exists a family $(U_n^p)_{p \in \N, n \in \Nc}$ such that $U_n^p \in \SSS_n$ for all $p \in \N$ and $n \in \Nc$, $\sum_{n \in \Nc} \nr{U_n^p}_\SSS^2 = 1$ for all $p \in \N$ but 
\[
\nr{\sum_{n \in \Nc} U_n^p }_\SSS \limt p {+\infty} 0.
\]
For $p \in \N$ we set 
\[
V_1^p = \sum_{n \in \Nc, n < N} U_n^p, \quad V_2^p = \sum_{n \geq N} U_n^p.
\]
After extracting a subsequence if necessary, we can assume that $V_1^p$ has a limit $V \in \mathsf{span}(\Phi_k)_{k \leq 2N}$ in $\SSS$. Since $V_1^p + V_2^p$ goes to 0, $V_2^p$ goes to $-V$, so $V$ also belongs to $\mathsf{span}(\Phi_k)_{k \geq 2N+1}$. By injectivity of $\Th$, this implies that $V = 0$. Then $U_n$ goes to 0 for all $n < N$, and hence
\[
\sum_{n = N}^{+\infty} \nr{U_n}_\SSS^2 \limt p {+\infty} 1.
\]
Since $V_2^p \to 0$, this gives a contradiction with \eqref{eq:almost-orthogonal-2}. Thus there exists $C > 0$ independant of the sequence $(U_n)$ such that 
\[
\nr{\sum_{n \in \Nc} U_n}_\SSS^2 \geq C\inv \sum_{n \in \Nc} \nr{U_n}_\SSS^2.
\]
With Proposition \ref{prop:intermediate-bases} we deduce that for $u \in \ell^2(\N^*)$ we have 
\[
\nr{\Th(u)}_\SSS^2 \gtrsim \nr{u}_{\ell^2(\N^*)}^2.
\]

\stepp To prove that $\Th$ is surjective we follow the proof of \cite[Theorem V.2.20]{kato}, except that the reference basis is not orthogonal. 
We denote by $\Th_0$ the map defined as $\Th$ with the family $(\Phi_k)_{k \in \N^*}$ replaced by $(\Phi_k^0)_{k \in \N^*}$ defined by \eqref{def:Phi-0} or \eqref{def:Phi-0-eq}. Since we already know that the family $(\Phi_k^0)_{k \in \N^*}$ is a Riesz basis, $\Th_0$ is boundedly invertible. Then $\Th \Th_0\inv - \Id_\SSS \in \Lc(\SSS)$ is the map 
\[
\Th \Th_0\inv - \Id_\SSS : \sum_{k=1}^\infty u_k \Phi_k^0 \mapsto \sum_{k=1}^\infty u_k (\Phi_k - \Phi_k^0).
\]
For $u = (u_k)_{k \in \N^*} \in \ell^2(\N^*)$ and $N \in \N$ we have 
\[
\nr{\sum_{k=N+1}^\infty u_k (\Phi_k - \Phi_k^0)}_\SSS^2 \leq \nr{u}_{\ell^2(\N^*)}^2 \sum_{k=N+1}^\infty \nr{\Phi_k - \Phi_k^0}_\SSS^2.
\]
By Proposition \ref{prop:DL}, \eqref{eq:lim-A2} and \eqref{eq:norm-etildez} we have 
\[
\nr{\Phi_k - \Phi_k^0}_\SSS^2 \lesssim \frac 1 {k^2},
\]
so $\Th \Th_0\inv-\Id$ is the limit in $\Lc(\SSS)$ of a family of finite rank operators, and hence it is compact. This implies that $\Th \Th_0\inv$ is a Fredholm operator. Since we already know that it is injective, it is surjective. Then $\Th$ is surjective and the proof is complete.
\end{proof}

\section{Resolvent for the Schr\"odinger operator on the wave guide} \label{sec-resolvent}

In this section we prove Theorem \ref{th-spectral-gap-Pc} for $a > 0$ and $b \in \R^*$. For this we deduce spectral properties of $\Pc = \Pc_{a,b}$ on $\HH$ from those of $\Tc = \Tc_{a,b}$ on $\SSS$. The intermediate result (Proposition \ref{prop:R} below) is valid for any $a,b \in \R$ with $(a,b) \neq (0,0)$ (in general we could proceed similarly, repeating the eigenvalues according to their geometric multiplicities, but the model case $a = b = 0$ is already clear).\\

We denote by $(\l_k)_{k \in \N^*}$ the sequence of (distinct) eigenvalues of $\Tc$. These eigenvalues have finite algebraic multiplicities and we know that for $k$ large enough the multiplicity of $\l_k$ is 1 if $a^2 \neq 4b^2$ and 2 if $a^2 = 4b^2$. In particular, if we denote by $m_k \in \N^*$ the multiplicity of the eigenvalue $\l_k$, we know that the sequence $(m_k)_{k \in \N}$ is bounded. We denote by $m$ its maximum.\\

For $u \in L^2(\R^{d-1})$ and $V \in \SSS$ we set $u\otimes V : (x,y) \in \O \mapsto u(x) V(y)$. Then, if we denote by $L$ the usual selfadjoint realization of the Laplacian on $\R^{d-1}$, we have $\Pc = L \otimes \Id_{L^2(\o)} + \Id_{L^2(\R^{d-1})} \otimes \Tc$.

Since the spectrum of $L$ is the half-line $[0,+\infty[$, it is known (see for instance Section XIII.9 in \cite{rs4}) that the spectrum of $\Pc$ is given by 
\[
\Sigma = \set{\l_k + r, k \in \N^*, r \in [0,+\infty[}.
\]

We can recover this fact directly in our context. Let $k \in \N^*$ and $r \geq 0$. We consider an eigenfunction $\Psi_k \in \SSS$ corresponding to the eigenvalue $\l_k$ of $\Tc$, and a sequence $(u_n)_{n \in \N}$ in $H^2(\R^{d-1})$ such that $\nr{u_n}_{L^2(\R^{d-1})} = 1$ for all $n \in \N$ and $(L - r) u_n \to 0$ in $L^2(\R^{d-1})$. Then for all $n \in \N$ the function $u_n \otimes \Psi_k$ belongs to $\Dom(\Pc)$ and in $\HH$ we have 
\[
\big( \Pc-(\l_k+r) \big) (u_n \otimes \Psi_k) = \big( (L - r) u_n \big) \otimes \Psi_k \limt {n}{+\infty} 0.
\]
This proves that $\l_k + r \in \s(\Pc)$, and hence $\Sigma \subset \s(\Pc)$. The converse inclusion will be a part of Proposition \ref{prop:R} below.\\

Let $k \in \N^*$. By Proposition \ref{prop:eq-z}, $\l_k$ is an eigenvalue of $\Tc$ of geometric multiplicity 1. There exist $\Psi_{k,1},\dots,\Psi_{k,m_k}$ such that $(\Tc - \l_k)\Psi_{k,1} = 0$ and $(\Tc-\l_k)\Psi_{k,j} = \Psi_{k,j-1}$ for $j \in \Ii 2 {m_k}$. By Proposition \ref{prop:Riesz-basis}, we can choose these vectors in such a way that $(\Psi_{k,j})_{k \in \N^*,1\leq j \leq m_k}$ is a Riesz basis of $\SSS$.\\

Let $F \in \HH \simeq L^2(\R^{d-1},\SSS)$. For almost all $x \in \R^{d-1}$ we have $F(x,\cdot) \in \SSS$, so there exist $f_{k,j}(x)$, $k \in \N^*$, $1 \leq j \leq m_k$, such that, in $\SSS$,
\begin{equation} \label{eq:F-fjk}
F(x,\cdot) = \sum_{k \in \N^*} \sum_{j=1}^{m_k} f_{k,j}(x) \Psi_{k,j}.
\end{equation}
Moreover, by the Riesz basis property, there exists $C \geq 1$ independant of $F$ such that, for almost all $x \in \R^{d-1}$,
\[
C\inv \sum_{k\in\N^*} \sum_{j=1}^{m_k} \abs{f_{k,j}(x)}^2 \leq \nr{F(x,\cdot)}_{\SSS}^2 \leq C \sum_{k\in\N^*} \sum_{j=1}^{m_k} \abs{f_{k,j}(x)}^2.
\]
After integration over $x \in \R^{d-1}$ we get 
\begin{equation} \label{eq:Riesz-basis-Omega}
C\inv \sum_{k\in\N^*} \sum_{j=1}^{m_k} \nr{f_{k,j}}_{L^2(\R^{d-1})}^2 \leq \nr{F}_{\HH}^2 \leq C \sum_{k\in\N^*} \sum_{j=1}^{m_k} \nr{f_{k,j}}_{L^2(\R^{d-1})}^2.
\end{equation}

With this notation we set, for $\z \in \C \setminus \Sigma$,
\begin{equation} \label{def:Rzeta}
\Rc(\z) F = \sum_{k = 1}^{+\infty} \sum_{j=1}^{m_k} \sum_{p = 0}^{m_k-j} (-1)^{p} \big( (L -(\z-\l_k) )^{-1-p} f_{k,j+p} \big) \otimes \Psi_{k,j}.
\end{equation}

\begin{proposition} \label{prop:R}
Let $\z \in \C \setminus \Sigma$. The series in the right-hand side of \eqref{def:Rzeta} converges in $\HH$ for all $F \in \HH$. This defines a bounded operator $\Rc(\z)$ on $\HH$ which is an inverse for $(\Pc-\z) : \Dom(\Pc) \to \HH$. Moreover, there exists $C > 0$ independant of $\z$ such that 
\begin{equation} \label{eq:norm-Rc-zeta}
\nr{\Rc(\z)}_{\Lc(\HH)} \leq \frac C {\mathsf{dist}(\z,\Sigma)^{[m]}}.
\end{equation}
\end{proposition}

\begin{proof}
Let $F \in \HH$. We use notation \eqref{eq:F-fjk}.

\stepp For $k \in \N^*$, $j \in \Ii 1 {m_k}$ and $p \in \Ii 0 {m_k-j}$ we have by the spectral theorem 
\begin{align*}
\nr{ (L -(\z-\l) )^{-1-p}}_{\Lc(\R^{d-1})} 
 \leq \frac 1 {\mathsf{dist}(\z-\l ,\R_+)^{1+p}} \leq \frac 1 {\mathsf{dist}(\z,\Sigma)^{1+p}}
 \leq \frac 1 {\mathsf{dist}(\z ,\Sigma)^{[m_k]}}.
\end{align*}
With \eqref{eq:Riesz-basis-Omega} we deduce that for $N,N_1 \in \N^*$ we have
\begin{multline*}
\nr{\sum_{k = N}^{N+N_1} \sum_{j=1}^{m_k} \sum_{p = 0}^{m_k-j} (-1)^{p} \big( (L -(\z-\l_k) )^{-1-p} f_{k,j+p} \big) \otimes \Psi_{k,j}}_\HH^2\\
\leq \frac {2mC}{(\mathsf{dist}(\z,\Sigma)^{[m]})^2} \sum_{k=N}^{N+N_1}  \sum_{j=1}^{m_k}  \abs{f_{k,j}}^2  \limt {N}{+\infty} 0,
\end{multline*}
so $\Rc(\z) F$ converges in $\HH$ and, by \eqref{eq:Riesz-basis-Omega} again,
\[
\nr{\Rc(\z) F}_{\HH} \lesssim \frac {\nr{F}_\HH} {\mathsf{dist}(\z ,\Sigma)^{[m]}}.
\]
This proves that $\Rc(\z)$ defines a bounded operator on $\HH$ and \eqref{eq:norm-Rc-zeta} is satisfied.

\stepp For $N \in\N^*$ we set 
\[
U_N = \sum_{k = 1}^{N} \sum_{j=1}^{m_k} \sum_{p = 0}^{m_k-j} (-1)^{p} \big( (L -(\z-\l_k) )^{-1-p} f_{k,j+p} \big) \otimes \Psi_{k,j}.
\]
Then $U_N \in \Dom(\Pc)$, it goes to $\Rc(\z)F$ in $\HH$ and 
\[
(\Pc-\z) U_N = \sum_{k=1}^N \sum_{j=1}^{m_k} f_{k,j} \otimes \Psi_{k,j} \limt N {+\infty} F.
\]
Since $\Pc$ is closed, this proves that $\Rc(\z) F \in \Dom(\Pc)$ and $(\Pc-\z) \Rc(\z) F = F$. In particular, for $\z_0 \in \rho(\Pc)$ we have
\begin{equation} \label{eq:Rc-Pcinv}
(\Pc-\z_0)\inv = \Rc(\z_0).
\end{equation}

\stepp It remains to prove that for $\z \in \C \setminus \Sigma$ and $U \in \Dom(\Pc)$ we have
\begin{equation} \label{eq:Rc-Pc}
{\Rc(\z) (\Pc-\z)U} = U.
\end{equation}
This will prove that $\z \in \rho(\Pc)$. Let $\z_0 \in \rho(\Pc)$. We can check that for $F \in \HH$ we have $\Rc(\z) \Rc(\z_0) F = \Rc(\z_0) \Rc(\z) F$. On the other hand we have 
\[
(\Pc-\z_0) \big( \Rc(\z) F - \Rc(\z_0) F \big) = (\z-\z_0) \Rc(\z) F
\]
so, by \eqref{eq:Rc-Pcinv}, $\Rc$ satisfies the resolvent identity
\[
\Rc(\z) F - \Rc(\z_0) F = (\z-\z_0) \Rc(\z_0) \Rc(\z) F =  (\z-\z_0) \Rc(\z) \Rc(\z_0) F.
\]
Applied with $F = (\Pc-\z_0)\inv U$ this gives
\[
\Rc(\z) (\Pc-\z) U = \Rc(\z) F  - (\z-\z_0) \Rc(\z) U = U.
\]
This proves \eqref{eq:Rc-Pc} and completes the proof.
\end{proof}

\section{More about the low frequency transverse eigenvalues} \label{sec-eigenvalues-2}

In Section \ref{sec-eigenvalues} we did not say much about the low frequency eigenvalues of the transverse operator $\Tc = \Tc_{a,b}$. We were quite accurate with the large eigenvalues, and then we said that the part of the spectrum in $D(0,R^2)$ (see Proposition \ref{prop:multiplicities-high-freq}) consists of a finite number of eigenvalues with finite multiplicities, and negative imaginary parts if $a > 0$ and $b \neq 0$. That was enough to prove that there is a spectral gap for $\Pc= \Pc_{a,b}$ and hence the local energy decay for \eqref{system}. In this final section we provide more information about these low frequency eigenvalues.\\

We consider $a \geq 0$ and $b \in \R$ such that $(a,b) \neq (0,0)$. Taking the adjoint, we get similar results for the case $a < 0$ (for the proofs the roles of $\phi_-$ and $\phi_+$ are reversed).\\

We recall from Proposition \ref{prop:eq-z} that the eigenvalues of $\Tc$ are geometrically simple. In the following proposition we discuss their algebraic simplicity.

\begin{proposition} \label{prop:multiplicities-z-z2}
Let $a \geq 0$ and $b \in \R$ such that $(a,b) \neq (0,0)$.
\begin{enumerate}[\rm (i)]
\item If $a^2 = 4 b^2$ then the eigenvalues of $\Tc$ are not simple.
\item Assume that $a^2 \neq 4b^4$. Let $\s \in \{\pm\}$ and let $z \in \C^*$ be a zero of $\phi_\s$. Then $z^2$ is an algebraically simple eigenvalue of $\Tc$ if and only if $z$ is a simple zero of $\phi_\s$.
\item All the zeros of $\phi_+$ are simple. If $a^2 > 4b^2$ then all the zeros of $\phi_-$ are simple. There exists a countably infinite subset $\Theta_+$ of $\set{(a,b) \in \R_+^* \times \R | a^2 < 4b^2}$ such that $\phi_-$ has a multiple zero if and only if $(a,b) \in \Theta_+$.
\end{enumerate}
\end{proposition}

\begin{proof}
\stepp If $a^2 = 4b^2$ then $(M-\mu)^2 = 0$ so, by \eqref{expr-ker-T-2}, $\ker((\Tc-z^2)^2)$ is at least of dimension 2 for any $z \in \Zc$.

\stepp Now we assume that $a^2 \neq 4b^2$. We consider $z \in \Zc$ and $\s \in \{\pm\}$ such that $\phi_\s(z)=0$. We have $\ker\big( (M-\mu_\s)^2 \big) = \ker(M-\mu_\s)$ so from \eqref{expr-ker-T-2} we see that $\ker\big((\Tc-z^2)^2\big) = \ker(\Tc-z^2)$ if and only if $\eta_\s(z) \neq 0$. On the other hand we have 
\[
\phi_\pm'(z) = e^{2iz\ell}-1 +2i\ell (z-\mu_\s) e^{2iz\ell}.
\]
Then we can check that $\eta_\s(z) = 0$ if and only if $\phi_\s' (z) = 0$. This gives the second statement.

\stepp If $\phi_+(z) = 0$ and $a > 0$ we have 
\begin{equation} \label{eq:mu-plus}
\Re \big(\mu_+ + i\ell(z^2-\mu_+^2) \big) = \Re(\mu_+) -\ell \Im(z^2) + \ell \Im(\mu_+^2) > 0,
\end{equation}
so $\eta_+(z) \neq 0$ and the zeros of $\phi_+$ are simple. If $a^2 > 4b^2$ then $\mu_-$ is real positive and we similarly see that the zeros of $\phi_-$ are simple. If $a = 0$ we have $z \in \R$ so
\[
\mu_+ + i\ell (z^2 - \mu_+^2) = i\abs b + i\ell z^2 + i\ell b^2 \neq 0
\]
and, again, the zeros of $\phi_+$ are simple.

\stepp Assume that there exists $z \in \Zc$ such that $\phi_-(z) = \phi_-'(z) = 0$. We assume for instance that $\Re(z) \geq 0$ and $\Im(z) \leq 0$. We have $e^{2iz\ell} + 1 \neq 0$,
\begin{equation} \label{expr:mu}
\mu_- = \frac {e^{2iz\ell} - 1}{e^{2iz\ell}  +1} z
\end{equation}
and
\begin{equation} \label{eq:sin-z}
\sin(2z\ell) + 2z\ell = 0.
\end{equation}
This last equality implies in particular that $z$ is not real or purely imaginary. We write $2z\ell = \xi + i \k$ with $\xi > 0$ and $\k < 0$. Then \eqref{eq:sin-z} gives
\begin{equation} \label{eq:system-ch-sin}
\begin{cases}
\cosh(\kappa) \sin(\xi) + \xi = 0,\\
\sinh(\kappa) \cos(\xi) + \kappa = 0.
\end{cases}
\end{equation}
In particular $\cos(\xi) < 0$ and $\sin(\xi) < 0$. Then
\begin{equation} \label{eq:argch-kappa}
\k =  -  \mathrm{argch}\left( -\frac {\xi} {\sin(\xi)} \right)
\end{equation}
and 
\begin{equation} \label{eq:sin-xi}
\cos(\xi) \sqrt {\frac {\xi^2}{\sin(\xi)^2}-1} +   \mathrm{argch}\left( -\frac {\xi} {\sin(\xi)} \right) = 0.
\end{equation}

\stepp Let $k \in \N$. The left-hand side of \eqref{eq:sin-xi} has a positive derivative in $I_k = \big](2k+1)\pi, (2k+\frac 32) \pi\big]$, it goes to $-\infty$ at $(2k+1)\pi$ and is positive at $(2k + \frac 32) \pi$.
We denote by $\xi_k$ the unique solution of \eqref{eq:sin-xi} in $I_k$. Then we define $\kappa_k$ by \eqref{eq:argch-kappa} and we set 
\begin{equation} \label{eq:z-xi-kappa}
z_k = \frac {\xi_k + i\kappa_k}{2\ell}.
\end{equation}
We finally define $\mu_k$ by \eqref{expr:mu}. Notice that if $\Re(\mu_k) \leq 0$ then there exist $\a \leq 0$ and $\b\in \R$ such that $z_k \in \Zc_{\a,\b}$ and $\mu_k = \mu_{+,\a,\b}$ or $\mu_k = \mu_{-,\a,\b}$. Since $\Im(z_k^2) < 0$ this gives a contradiction and proves that $\Re(\mu_k) > 0$. Then we set 
\[
a_k = 2 \Re(\mu_k) >0 \quad \text{and} \quad b_k = \frac 12  \sqrt{a_k^2 + \Im(\mu_k)^2} > 0.
\]

\stepp If $a > 0$ and $b \in \R^*$ are such that $a^2 < 4b^2$ and $\Tc_{a,b}$ has an eigenvalue with algebraic multiplicity greater than 1, then this eigenvalue is necessarily $z_k^2$ for some $k \in \N$, we have $\mu_{-,a,b} = \mu_k$ and, finally, $a = a_k$ and $\abs b = b_k$. Conversely, $\Tc_{a_k,b_k}$ and $\Tc_{a_k,-b_k}$ have an eigenvalue $z_k^2$ of algebraic multiplicity greater than 1 for all $k \in \N$. This proves that $\Tc_{a,b}$ has a non-simple eigenvalue if and only if $(a,b)$ belongs to $\bigcup_{k \in \N} \{(a_k,b_k),(a_k,-b_k)\}$ and concludes the proof.
\end{proof}

For a single equation with damping at the boundary, is is proved in \cite{royer-diss-schrodinger-guide} that for each $n \in \N$ there is exaclty one square root of an eigenvalue with real part in $]n\nu,(n+1)\nu[$. This came from the continuity of the spectrum with respect to the absorption index $a$ and the fact that the square roots of the eigenvalues do not have their real parts in $\nu \N$. This ensured in particular that each eigenvalue of the transverse operator is simple. Here we prove an analogous result, but the situation is not as simple as it was for a single equation.

\begin{proposition} \label{prop:zero-Nnu}
Let $a \geq 0$ and $b \in \R$ such that $(a,b) \neq (0,0)$. Let $n \in \N^*$. 
\begin{enumerate}[\rm (i)]
\item If $\vf_+(z) = 0$ then $\Re(z) \neq n\nu$.

\item If $b = 0$ (and $a > 0$) we have $\vf_-(n\nu) = 0$ (this is a simple zero) and $\vf_-$ has no other zero of real part equal to $n \nu$.

\item If $b \neq 0$ then $\vf_-$ has a zero of real part $n\nu$ if and only if $0<a^2 < 4b^2$, $a < 2n\nu$ and 
\begin{equation} \label{hyp-vp-nnu}
e^{2n\pi \sqrt{\frac {4b^2}{a^2} -1 }} = \frac {2n\nu + a}{2n \nu - a}.
\end{equation}
In this case $z$ is unique and is given by 
\begin{equation} \label{eq:z-nnu}
z = n\nu - i n \nu \sqrt{\frac {4b^2}{a^2} - 1}.
\end{equation}
This is a simple zero of $\phi_-$.
\end{enumerate}
\end{proposition}

\begin{proof}
\stepp We easily see that $n\nu$ belongs to $\Zc$ if and only if $b = 0$, and that in this case it is a simple zero of $\phi_-$ and not a zero of $\phi_+$.

\stepp Assume that $z \in \Zc$ satisfies $\Re(z) = n\nu > 0$ and $\Im(z) \neq 0$. We necessarily have $a > 0$ and $\Im(z) < 0$ so
\[
e^{2iz\ell} = e^{-2\ell \Im(z)}  \in ]1,+\infty[.
\]
We set 
\[
\kappa =  \frac {e^{2iz\ell}+1}{e^{2iz\ell}-1} \quad \in ]1,+\infty[.
\]
By \eqref{eq:exp-2izl-bis} we have $z = \k \mu_+$ or $z = \k \mu_-$. 
If $a^2 \geq 4 b^2$ then $\mu_+$ and $\mu_-$ are real, which gives a contradiction. Thus we have $a^2 < 4b^2$.
Since $\Im(\mu_+) > 0$ we cannot have $z = \k \mu_+$. Therefore $z = \k \mu_-$. Since $\Re(\mu_-) = \frac a 2$ we necessarily have 
\begin{equation} \label{eq:kappa}
\k = \frac {2n\nu} a,
\end{equation}
which gives \eqref{eq:z-nnu} and implies $2n\nu > a$. Then we can write 
\begin{equation} \label{eq:f-z}
e^{2n\pi \sqrt{\frac {4b^2}{a^2} -1 }} - \frac {2n\nu+a}{2n\nu-a} = e^{2iz\ell} -  \frac {\k  +1}{\k -1} = 0.
\end{equation}

Conversely, if \eqref{hyp-vp-nnu} holds then with $z$ defined by \eqref{eq:z-nnu} we have $z = \k \mu_-$ with $\kappa$ given by \eqref{eq:kappa}, so the equality \eqref{eq:f-z} now gives $\phi_-(z) = 0$, and hence $z \in \Zc$.

Finally, for $z$ given by \eqref{eq:z-nnu} we have
\[
\Im\big(\phi_-'(z)\big) = \Im \big(  e^{2iz\ell} - 1  +2i\ell\mu_-(\k-1) e^{2iz\ell} \big) = 2\ell (\k-1) e^{2iz\ell} \Re(\mu_-) \neq 0, 
\]
so $\phi_-'(z)\neq 0$ and $z$ is a simple zero of $\phi_-$
\end{proof}

For $n \in \N$ we set $\C_n = \set{z \in \C \, : \, n\nu<\Re(z)<(n+1)\nu}$. We recall that if $a^2 \neq 4b^2$ then the functions $\phi_-$ and $\phi_+$ have no common zeros.

\begin{proposition} \label{prop:number-zeros-Cn}
Let $a \geq 0$ and $b \in \R$. Let $n \in \N$.
\begin{enumerate}[\rm(i)]
\item If $a > 0$ and $b = 0$ then $\phi_-(n\nu) = 0$ (this is a simple zero), $\phi_-$ has no zero in $\C_n$ and $\phi_+$ has a unique zero in $\C_n$ (which is also simple).
\item If $a^2 > 4b^2 > 0$ then $\phi_\pm$ has a unique zero $\z_{n,\pm}$ in $\C_n$ (and this zero is simple).
\item If $a^2 = 4b^2 > 0$ then $\phi_+ = \phi_-$ has a unique zero $\z_{n}$ in $\C_n$, and this zero is simple (therefore it is a zero of order 2 for the product $\phi_- \phi_+$).
\item If $a^2 < 4b^2$ then the function $\phi_+$ has a unique zero $\z_{n,+}$ in $\C_n$ (and this zero is simple). There exists a unique $\th \in \big] \frac a{2\nu} , +\infty \big]$ such that
\begin{equation} \label{def:theta}
4 b^2 = a^2  + \frac {a^2} {4\th^2\pi^2} \ln\left( \frac {2\th \nu + a}{2\th\nu-a} \right)^2,
\end{equation}
and then the number of zeros of $\phi_-$ in $\C_n$ (counted with multiplicities) is 
\[
\begin{cases}
2 & \text{if } n < \th < n+1,\\
1 & \text{otherwise}.
\end{cases}
\]
\end{enumerate}
\end{proposition}

\begin{proof} 
\stepp We begin with the case $a^2 > 4b^2 > 0$. By Proposition \ref{prop:Sigma-disks}, $\Zc_{sa,sb}$ is included in a horizontal strip of $\C$ which does not depend on $s\in [0,1]$. Moreover, by Proposition \ref{prop:zero-Nnu}, $\Zc_{sa,sb}$ does not intersect the vertical lines $\Re(\z) = n\nu$ and $\Re(\z) = (n+1)\nu$ for $s \in ]0,1]$. Then, by the Rouch\'e Theorem, the number of zeros of $\phi_{\pm,s} = \phi_{\pm,sa,sb}$ in $\C_n$ is a continuous and hence constant function of $s \in ]0,1]$.

Assume that $n \neq 0$. We know that $n\nu$ is a simple zero of $\phi_{0}$ (see \eqref{def:phi0}). By the implicit functions theorem there exist $s_0 \in ]0,1]$, a neighborhood $\Vc$ of $n\nu$ in $\C$ and analytic functions $\z_{n,\pm} : [0,s_0] \to \Uc$ such that for $s \in [0,s_0]$ and $\z \in \Vc$ we have $\phi_{\pm,s}(\z) = 0$ if and only if $\z = \z_{n,\pm}(s)$. We can compute 
\begin{equation} \label{eq:zeta-n}
\z_{n,\pm}(s) = n\nu - \frac {i\mu_\pm s}{n\pi} + \frac {\mu_\pm^2 \ell s^2}{n^3 \pi^3} + \bigo s 0 (s^3).
\end{equation}
In particular, for $s > 0$ small enough we have $\z_{n,\pm}(s) \in \C_n$. Moreover, by Proposition \ref{prop:multiplicities-z-z2} applied to $\Tc_s = \Tc_{sa,sb}$, the eigenvalues $\z_{n,\pm}(s)^2$ of $\Tc_s$ are simple (we can also abserve that $n^2\nu^2$ has multiplicity 2 for $\Tc_0$ and splits into two distinct eigenvalues $z_{n,\pm}(s)^2$ of $\Tc_s$, so each of these two eigenvalues is necessarily simple for $s > 0$ small).

We proceed similarly around $(n+1)\nu$, and we see that the zero $(n+1)\nu$ of $\phi_{\pm,0}$ moves to $\C_{n+1}$ for $s > 0$ small. Finally, for $s > 0$ small the functions $\phi_{\pm,s}$ have exactly one simple zero in $\C_n$. We recall that by Proposition \ref{prop:eq-z} applied to $\Tc_s$, the zeros of $\phi_{+,s}$ and $\phi_{-,s}$ cannot meet. Then the functions $\phi_{\pm,s}$ have exactly one simple zero in $\C_n$ for any $s \in ]0,1]$ .
All this holds in particular for $s = 1$ and the first statement is proved for $n\neq 0$.

We proceed similarly for $n = 0$. 0 is a zero of $\phi_{\pm,0}$ of multiplicity 2, which splits into two opposite simple zeros for $s > 0$ small (once squared, they correspond to the same eigenvalue of $\Tc_s$). One of these two zeros is in $\C_0$, and we conclude similarly. 
 
\stepp When $a > 0$ and $b = 0$ we know that $n\nu$ is a zero of $\vf_{-,s}$ for all $s \in [0,1]$ and the zero of $\vf_{+,s}$ near $n\nu$ behaves as in the previous case.

\stepp We continue with the case $a^2 = 4b^2$. We set $\phi_s = \phi_{\pm,s}$. 
As above, if $n \neq 0$ then $n\nu$ is a simple zero of $\phi_0$, and for $s > 0$ small there is a unique zero $\z_n$ of $\phi_s$ near $n\nu$, and \eqref{eq:zeta-n} holds for $\z_n(s)$ with $\mu_\pm = \frac a 2$. 
Then for $s > 0$ small the function $\phi_s$ has a unique zero in $\C_n$. Since this is a simple zero there is no splitting, we have a unique simple zero in $\C_n$ for all $s \in ]0,1]$. Then, by continuity of the spectrum of $\Tc_s$ with respect to $s$, the corresponding eigenvalue $\z_n(s)^2$ of $\Tc_s$ has algebraic multiplicity 2, as is the case for the eigenvalue $n^2 \nu^2$ of $\Tc_0$. As above we deal similarly with the case $n = 0$.

\stepp We finish with the case $a^2 < 4b^2$. The zeros of $\phi_{s,+}$ behave exactly as in the first case: the asymptotic expansion \eqref{eq:zeta-n} still holds and, since $\mu_+$ has a positive imaginary part, we still have $\z_n \in \C_n$ for $s > 0$ small. The same applies to the zeros of $\phi_{-,s}$ if $a = 0$ and $b \neq 0$.

\stepp Things are different for the zeros of $\phi_{s,-}$ in general, since by Proposition \ref{prop:zero-Nnu} they can go through the vertical lines $\Re(z) \in \nu \N$. Assume that $a > 0$. For $s \in [0,1]$ we set 
\[
a(s) = sa, \quad b(s) = \frac {sa\e_b}2  \sqrt{1 + \frac 1 {4\th^2\pi^2} \ln\left( \frac {2\th \nu + sa}{2\th\nu-sa} \right)^2},
\]
where $\e_b \in \{\pm 1\}$ is the sign of $b$ and $\theta$ is defined by \eqref{def:theta}. We also set $\mu(s) = \mu_{-,a(s),b(s)}$ and, for $z \in \C$, $\tilde \phi_s(z) = \phi_{-,a(s),b(s)}(z)$. We have 
\[
\mu(s) = \frac {sa} 2 - \frac {is^2 a^2}{4\th^2 \pi \nu} + \Oc(s^3).
\]
We set $\tilde \Tc_s = \Tc_{a(s),b(s)}$. Then we can proceed as above. If $n \neq 0$, $n\nu$ is a simple zero of $\tilde \phi_0$. By the implicit function Theorem there exist $s_0 \in ]0,1]$, a neighborhood $\Vc$ of $n\nu$ and an analytic function $\tilde \z_n : [0,s_0] \to \Vc$ such that for $s \in [0,s_0]$ and $\z \in \Vc$ we have $\tilde \phi_s(\z) = 0$ if and only if $\z = \tilde \z_n(s)$. Moreover, 
\[
\tilde \z_n(s) = n\nu - \frac {isa}{2n\pi} + \frac {s^2 a^2(\th^2 - n^2)}{4\pi^2 \nu n^3 \th^2} + \bigo s 0 (s^3).
\]
We see that for $s>0$ small enough we have $\tilde \z_n(s) \in \C_n$ if $\th > n$ and $\tilde \z_n(s) \in \C_{n-1}$ if $\th < n$. If $\th = n$, we necessarily have $\Re(\tilde \z_n(s)) = n\nu$ by Proposition \ref{prop:zero-Nnu}.

We proceed similarly around $(n-1)\nu$ and $(n+1)\nu$, and we obtain that for $s > 0$ small the numbers of zeros of $\tilde \phi_s$ with real part in $](n-1)\nu,n\nu[$, $\set {n\nu}$ and $]n\nu,(n+1)\nu[$ depend on the value of $\theta$ as follows.
\begin{equation} \label{eq:tabular}
\begin{tabular}{|c|c|c|c|}
\cline{2-4}
\multicolumn{1}{c}{}&\multicolumn{3}{|c|}{$\Re(\z) \in $}\\
\cline{2-4}
\multicolumn{1}{c|}{} & $](n-1)\nu,n\nu[$ & $\{n\nu\}$ & $]n\nu, (n+1)\nu[$\\
\hline
$\th \leq  n-1$ &  1 & 0 & 1 \\
\hline 
$n-1< \th < n$ & 2 & 0 & 1 \\
\hline 
$\th = n$ & 1 & 1 & 1 \\
\hline 
$n< \th < n+1$ & 1 & 0 & 2 \\
\hline 
$ \th \geq  n+1$ & 1 & 0 & 1\\
\hline
\end{tabular}
\end{equation}

By continuity of the zeros of $\tilde \phi_s$, we can extend this observation for all $s \in [0,1]$. If $\th \neq n$, then the zeros of $\tilde \phi_s$ cannot go through the line $\Re(\z) = n\nu$. If $\th = n$, there is always one zero of $\tilde \phi_s$ on this line. Since this zero is always simple, the zeros in $\C_{n-1}$ or $\C_n$ cannot meet it. Since the same argument is valid for the lines $\Re(\z) = (n-1)\nu$ and $\Re(\z) = (n+1)\nu$, we obtain that \eqref{eq:tabular} holds for any $s \in ]0,1]$. The only difference is that when we have two zeros (for instance in $\C_n$ if $\th \in ]n\nu,(n+1)\nu[$) then for $s$ small we know that we have two simple zeros (one close to $n\nu$ and the other close to $(n+1)\nu$), while for large $s$ they could meet and produce a zero of multiplicity 2 (and this happens for some values of $(a,b)$ by Proposition \ref{prop:multiplicities-z-z2}). The discussion of the case $n = 0$ is the same as above (in this case we necessarily have $\th > n$).
\end{proof}

\end{document}